\documentclass[10pt]{amsart}
\usepackage{amscd,graphicx,epsfig}
\usepackage{amssymb}
\usepackage{mathrsfs}
\usepackage[colorlinks=true, pdfstartview=FitV, linkcolor=blue, citecolor=blue, urlcolor=blue, bookmarks=false]{hyperref}
\usepackage[all]{xy}
\usepackage{enumerate}
\usepackage{mathabx}

\title[Automorphisms of the plane preserving a curve]
{Automorphisms of the plane preserving a curve}
\author{J\'er\'emy Blanc And Immanuel Stampfli}

\address{Mathematisches Institut,
Universit\"at Basel, Rheinsprung 21, CH-4051 Basel}
\email{Jeremy.Blanc@unibas.ch}

\address{Mathematisches Institut,
Universit\"at Basel, Rheinsprung 21, CH-4051 Basel}
\email{Immanuel.E.Stampfli@gmail.com}

\thanks{The authors gratefully acknowledge support by the Swiss National Science Foundation Grants ``Birational Geometry" 128422 and ``Automorphisms of Affine n-Space" 137679.}

\theoremstyle{plain}
\newtheorem{thm}{Theorem}
\newtheorem{prop}{Proposition}[section]
\newtheorem{lem}[prop]{Lemma}
\newtheorem{cor}[prop]{Corollary}

\theoremstyle{definition}
\newtheorem{defn}[prop]{Definition}
\newtheorem{exa}[prop]{Example}

\theoremstyle{remark}
\newtheorem{rem}[prop]{Remark}

\newcommand{\name}[1]{\textsc{#1\/}}
\renewcommand{\AA}{{\mathbb A}}
\newcommand{\A}{{\mathbb A}}

\newcommand{\PP}{\mathbb P}
\newcommand{\p}{\mathbb P}
\newcommand{\Gal}{\mathrm{Gal}}
\newcommand{\K}{\mathbf k}
\newcommand{\FF}{\mathbb F}
\newcommand{\GG}{\mathbb G}

\newcommand{\bir}{\dashrightarrow}

\DeclareMathOperator{\h}{ht}
\DeclareMathOperator{\len}{len}

\DeclareMathOperator{\Aut}{Aut}
\DeclareMathOperator{\Bir}{Bir}

\DeclareMathOperator{\Aff}{Aff}
\DeclareMathOperator{\GL}{GL}
\DeclareMathOperator{\PGL}{PGL}
\DeclareMathOperator{\J}{J}

\newcommand{\KK}{\overline{{\mathbf k}}}
\DeclareMathOperator{\id}{id}

\renewcommand{\char}{\textrm char}
\renewcommand{\H}{H}

\newcommand{\tr}{\textcolor{black}}
\newcommand{\lab}[1]{\label{#1}}

\begin{document}

\begin{abstract}
We study the group of automorphisms of the affine plane preserving some given curve, over any field. The group is proven to be algebraic, except in the case where the curve is a bunch of parallel lines. Moreover,
a classification of the groups of positive dimension occuring is also given in the case where the curve is geometrically irreducible and the field is perfect.
\end{abstract}

\maketitle

\section{Introduction}
\label{Intro.sec}
Let $\K$ be an arbitrary field. This article studies (closed) curves $\Gamma\subset\A^2=\A^2_\K$ and the group of automorphisms of $\A^2$ (defined over~$\K$) which preserve this curve. {We will denote this group} by $\Aut(\A^2,\Gamma)$. In other words, we study polynomials in $\K[x,y]$ and the $\K$-algebra automorphisms of $\K[x,y]$ that send the polynomial on a multiple of itself.
We will always assume that the curve is reduced, i.e.\;that the polynomial does not contain any multiple factor. For our purpose, this is a natural assumption.

If $\Gamma$ has equation in $\K[x]$, we will say that $\Gamma$ is a \emph{fence}.
{In this case}, $\Aut(\A^2,\Gamma)$ is easy to describe, and {it} is in fact a countable union of algebraic groups. Our main result consists in showing that this is the only case where such phenomenon occurs.

\medskip

Recall that $\Aut(\A^2)$ has the structure of an ind-variety. More precisely, the set $\Aut(\A^2)_{d}$ of automorphisms of degree $\le d$ is an algebraic variety and $\Aut(\A^2)_{d}$ is closed in $\Aut(\A^2)_{d+1}$ for any $d$. This gives to $\Aut(\A^2)=\bigcup_{d=1}^{\infty} \Aut(\A^2)_d$ the structure of an ind-variety, and since composition and taking inverse preserve this structure, $\Aut(\A^2)$ is an ind-group (see \cite[Chapter IV]{Kum02} for precise definitions of ``ind-variety" and ``ind-group").

 By definition, an algebraic subgroup of $\Aut(\A^2)$ is a closed subgroup of bounded degree.
Let us recall that the group $\Aut(\A^2)$ contains the following natural algebraic subgroups
\[
	\begin{array}{rcl}
		\Aff(\A^2)&=&\{(x,y)\mapsto (ax+by+e,cx+dy+f)\ |\ a,b,c,d,e,f\in \K, \\
				&& \ ad-bc \not=0\}, \linebreak
		\medskip \\
		\J_n&=&\{(x,y)\mapsto (ax+P(y),by+c)\ |\ a,b\in \K^{*}, c\in \K, P\in \K[y], \\
		 	  && \ \deg(P)\le n\}.
	\end{array}
\]
Moreover, $\Aut(\A^2)$ is generated by the union of these groups (Jung - van der Kulk's Theorem \cite{Jun42}, \cite{VdK53}), and any algebraic subgroup of $\Aut(\A^2)$ is conjugated to a subgroup of $\Aff(\A^2)$ or $\J_n$ for some $n$ \cite[Theorem 4.3]{Kam79}.

\medskip

For any curve $\Gamma\subset \A^2$, the group $\Aut(\A^2,\Gamma)$ is a closed ind-subgroup of $\Aut(\A^2)$. The following result describes when $\Aut(\A^2,\Gamma)$ is an algebraic group, i.e.\;when it has bounded degree.


\begin{thm}\label{Thm:Main}
Let $\Gamma$ be a  curve in $\A^2$. Applying an automorphism of $\A^2$, one of the following holds:

\begin{enumerate}[$i)$]
\item
The curve $\Gamma$ has equation  $F(x)=0$, where $F(x)\in \K[x]$ is a square-free
polynomial and 
\[
	\begin{array}{rcl}
	\Aut(\A^2,\Gamma) &=& \{(x,y)\mapsto (ax+b,cy+P(x))\mid a,c\in \K^{*} , \, b \in \K \, , \, \\
	&& \ P\in 	\K[x] \, , \ F(ax+b)/F(x)\in \K^{*}\} \, .
	\end{array}
\]
\item
The group $\Aut(\A^2,\Gamma)$ is equal to 
\[
	\{g\in \Aff(\A^2) \mid g(\Gamma)=\Gamma\}
\quad
\textrm{or} 
\quad
	\{g\in \J_n\mid g(\Gamma)=\Gamma\}
\] 
for some integer $n$. Moreover, the action of  $\Aut(\A^2,\Gamma)$ on $\Gamma$ gives an isomorphism of $\Aut(\A^2,\Gamma)$ with a closed subgroup of $\Aut(\Gamma)$, $($this latter being an algebraic group$)$.
\end{enumerate} 
In particular, $\Aut(\A^2,\Gamma)$ is an algebraic group if and only if there is no automorphism of $\A^2$ which sends $\Gamma$ onto a fence.
\end{thm}

Remark that the existence of an automorphism of $\A^2$ which sends $\Gamma$ onto a line implies that $\Gamma\simeq \A^1$, and that the converse is true {in} characteristic $0$ by the Abhyankar-Moh-Suzuki Theorem (\cite{Suz74}, \cite{AbMo75}), but false in general (last page of \cite{Nag71}). So Theorem~\ref{Thm:Main} implies, in positive characteristic, that non-trivial embeddings of $\A^1$ into $\A^2$ are rigid in the sense that the image of the curve is preserved by only a few automorphisms of $\A^2$, namely an algebraic group.

Another consequence of Theorem~\ref{Thm:Main} is that the fixed locus of an automorphism of $\A^2$ only contains points and curves equivalent to lines, a result already observed by Friedland and Milnor in \cite{FM89}, in the case where $\K=\mathbb{C}$ (see \cite{Jel03} for some generalisations to higher {dimensions}). 

Remark that in the case where $\K=\mathbb{C}$, the observation on fixed points of Friedland and Milnor, the  Abhyankar-Moh-Suzuki Theorem and the Lin-Zaidenberg Theorem \cite{ZL83} imply Theorem~\ref{Thm:Main} in the case where the curve 
$\Gamma$ has an irreducible component non-isomorphic to $\mathbb{C}^{*}$. 
The interesting case of Theorem~\ref{Thm:Main}, for $\K = \mathbb{C}$, is thus the description of $\Aut(\A^2,\Gamma)$ when $\Gamma$ is isomorphic to $\mathbb{C}^{*}$ (or a union of such curves). Note that there exist only partial classifications of the closed embeddings of $\mathbb{C}^{*}$ into $\A^2$, which are moreover very involved 
(see \cite{CNKR09}, \cite{BoZo10} and more recently \cite{KPR14}).
Hence, one cannot  check a few cases to derive Theorem~\ref{Thm:Main}.
Moreover, note
that there exist complicated torsion-abelian closed subgroups of $\Aut(\A^2)$ which are not conjugated to a subgroup of $\Aff(\A^2)$ or to a subgroup of 
the de Jonqui\`eres group $\bigcup_{n=1}^{\infty} \J_n$ (see \cite{Wri79}). A priori, such groups could preserve curves isomorphic to $\A^1\setminus\{0\}$, but
Theorem~\ref{Thm:Main} implies that these groups do not preserve any algebraic curve.

\medskip

The proof of Theorem~\ref{Thm:Main} is done in Section~\ref{Sec:ProofThm}, using tools of birational geometry of surfaces introduced in Section~\ref{Sec:RemBir}.

In Section~\ref{Sec:Class}, we refine the theorem by describing more precisely the possibilities for the group $\Aut(\A^2,\Gamma)$, in the case where $\Gamma$ is geometrically irreducible and the ground field $\K$ is perfect. We obtain then the following result (the tori $T$ of cases $iv)$, $v)$ are precisely described in Proposition~\ref{Prop:Torus}).
\begin{thm}\label{Thm:Class}
Let $\Gamma$ be a geometrically irreducible closed curve in $\A^2$, defined over a perfect field $\K$. Applying an automorphism of $\A^2$, one of the following holds:

\begin{enumerate}[$i)$]
\item
The curve $\Gamma$ is the line {with} equation $x=0$ and 
\[
	\Aut(\A^2,\Gamma)=\{(x,y)\mapsto (ax,by+P(x))\mid a,b\in \K^{*} , \, P\in \K[x]\} \, .
\]
\item
The curve $\Gamma$ has equation {$x^b=\lambda y^a$}, where $\lambda\in \K^{*}$ and $a,b> 1$ are coprime integers. Moreover, $\Aut(\A^2,\Gamma)$ is equal to the group $\K^{*}$ acting diagonally via $(x,y)\mapsto (t^ax ,t^by)$.
\item
The curve $\Gamma$ has equation $x^by^a=\lambda$, where $\lambda\in \K^{*}$ and $a,b\ge 1$ are coprime integers. Moreover, $\Aut(\A^2,\Gamma)$ contains the group $\K^{*}$ acting diagonally via $(x,y)\mapsto (t^ax ,t^{-b}y)$, and is equal to this group if $(a,b)\not=(1,1)$, or is the group $\K^{*}\rtimes \mathbb{Z}/2\mathbb{Z}$ generated {by $\K^\ast$ and by} $(x,y)\mapsto (y,x)$ if $(a,b)=(1,1)$.
\item
The curve $\Gamma$ has equation $\lambda x^2+\nu y^2=1$, where 
			$\lambda,\nu\in \K^{*}$, $-\lambda\nu$ is not a square in $\K$ and $\char(\K)\not=2$. The group $\Aut(\A^2,\Gamma)$ is the subgroup of $\GL(2,\K)$ preserving the form $\lambda x^2+\nu y^2$. It is isomorphic to $T\rtimes \mathbb{Z}/2\mathbb{Z}$ for some non-$\K$-split torus $T$.\item
The curve $\Gamma$ has equation  $ x^2+\mu xy+y^2=1$, where 
			$\mu \in \K^{*}$ and $x^2+\mu x+1$ has no root in $\K$, and $\char(\K)=2$. The group $\Aut(\A^2,\Gamma)$ is the subgroup of $\GL(2,\K)$ preserving the form $x^2+\mu xy+y^2$. It is isomorphic to $T\rtimes \mathbb{Z}/2\mathbb{Z}$ for some non-$\K$-split torus $T$.
\item
The group $\Aut(\A^2,\Gamma)$ is a zero-dimensional $($hence finite$)$ subgroup of $\Aff(\A^2)$ or $\J_n$ for some $n$.

\end{enumerate} 
\end{thm}
A similar description was also done in the case $\K=\mathbb{C}$ and where the curve is connected and simply connected in \cite{AZ13}, and more generally for such curves on affine toric surfaces. The intersection of the two studies give the cases $i)$ and $ii)$ above.

\bigskip

We would like to thank the referee for his helpful remarks and corrections.

\section{Reminders on completions of $\A^2$}\label{Sec:RemBir}
\label{RemCompletions.sec}
In this section, we define natural completions of $\AA^2$, and 
links between them.
\subsection{Natural completions of $\A^2$}
\label{NatCompletions.subsec}

\begin{exa}
	\label{Exa:P2}
	The  morphism
	\[
		\begin{array}{rcl}
			\AA^2&\hookrightarrow &\PP^2\\
			(x,y)&\mapsto &(x:y:1)
		\end{array}
	\]
	yields an isomorphism 
	$\AA^2\stackrel{\sim}{\to} \PP^2\backslash L_{\p^2}$, where 
	$L_{\p^2}$ is the line of $\PP^2$ with equation $z=0$.
\end{exa}

\begin{exa}
	\label{Exa:Hirzebruch}
	For $n \ge 1$, the $n$-th Hirzebruch surface $\FF_n$ is
	\[
		\FF_n=\{((a:b:c),(u:v))\in \PP^2 \times \PP^1\ |\ bv^n=cu^n\}.
	\]
	Let $E_n, L_{\mathbb{F}_n} \subset \FF_n$ be the curves given 
	by $(1:0:0) \times \PP^1$ and $v=0$, respectively.
	The  morphism
	\[
		\begin{array}{rcl}
			\AA^2&\hookrightarrow & \FF_n \\
			(x,y)&\mapsto &((x:y^n:1),(y:1))
		\end{array}
	\]
	gives an isomorphism 
	$\AA^2\stackrel{\sim}{\to} \FF_n \backslash (E_n\cup L_{\FF_n})$.
	Note that $E_n$ is the unique section of $\pi\colon \FF_n\to \p^1$ of 
	self-intersection $-n$, and $L_{\FF_n}$ {has  
	self-intersection $0$ since it is a smooth fibre}.
\end{exa}

\begin{rem}
	In Example~\ref{Exa:Hirzebruch}, we could also have chosen 
	$n=0$, which yields the surface $\FF_0$, isomorphic to 
	$\PP^1\times \PP^1$, via  $((x:y:z),(u:v))\mapsto ((x:y),(u:v))$, 
	but we will not need this one in the sequel.
\end{rem}

\begin{defn}
	A \emph{natural completion} (of $\AA^2$) 
	is a pair $(X,B)$ 
	which is given in Example~\ref{Exa:P2} or Example~\ref{Exa:Hirzebruch}: 
	either $(X,B)=(\PP^2,L_{\p^2})$ or 
	$(X,B)=(\FF_n,E_n\cup L_{\FF_n})$ for some $n\ge 1$. 
	The isomorphism $\A^2\to X\backslash B$ given in the examples will be 
	called \emph{canonical isomorphism}, or 
	\emph{canonical embedding of $\A^2$ into $X$}.

	A birational map (respectively birational morphism) $(X,B) \bir (X',B')$ 
	between two natural completions is a birational map 
	(respectively birational morphism) $X \bir X'$ inducing an 
	isomorphism $X\backslash B\to X'\backslash B'$.

	We denote by $\Aut(X,B)$ the group of automorphisms of 
	$(X,B)$: it is the group of automorphisms of $X$ which leave 
	$B$  (or equivalently $X\setminus B=\AA^2$) invariant.
	
\end{defn}
\begin{rem}
Given two natural completions $(X,B)$, $(X',B')$, any birational map $\varphi\colon (X,B)\dasharrow (X',B')$ restricts to an isomorphism $X\backslash B\to X'\backslash B'$, and corresponds thus, via the canonical isomorphisms $\A^2\simeq X\backslash B$ and $\A^2\simeq X'\backslash B'$, to {a unique} 
automorphism of $\A^2$. Moreover, every automorphism of $\A^2$ is
obtained {in this way.}
\end{rem}

Let us recall some easy fact on the automorphisms of natural completions:
\begin{lem}
\lab{LEM:AutoNatCompl}
Let $(X,B)$ be a natural completion of $\A^2$, and let  $\iota\colon \A^2 \to X\setminus B$ be the associated canonical 
isomorphism. The group $\Aut(X,B)$ 
corresponds via $\iota$ to the following subgroups of $\Aut(\AA^2)$ $($see the introduction for the definition of $\Aff(\A^2)$ and $\J_n)$:

$1)$ If $X=\p^2$, then $\Aut(X,B)=\Aff(\A^2)\simeq \GL(2,\K)\ltimes \K^2$

$2)$ If $X=\FF_n$, then 
	$\Aut(X,B)=\J_n\simeq (\K^\ast)^2 \ltimes (\K \ltimes \K^{n+1}) $.
\end{lem}
\begin{proof}
The first assertion follows from the fact that $\Aut(\p^2)$ is the group of linear automorphisms, equal to $\PGL(3,\K)$.

Recall that for $n \ge 1$ the $\p^1$-bundle structure on the Hirzebruch surface
$\FF_n$ given by the projection on the second factor is unique. Hence, the group of automorphisms of $\FF_n$ preserves the $\p^1$-bundle structure. In particular, {the restriction of each  automorphism of 
$\Aut(\FF_n, E_n \cup L_{\FF_n})$ is an automorphism of $\A^2$ that
preserves the fibration $\AA^2 \to \AA^1, (x,y)\mapsto y$}. The precise description of the degree follows from a straightforward calculation.
\end{proof}

\subsection{Elementary links}
\label{ElemLinks.subsec}
There are two kinds of very simple birational maps between natural completions. The first one are automorphisms, that we described in Lemma~\ref{LEM:AutoNatCompl}, and the second one
are elementary links, that we describe now. The results of this section are classical, we just remind them to the reader for self-containedness and in order to fix notation.

\begin{defn} \label{Def:Elmlinks}
	A \emph{link} is a birational map 
	$\varphi\colon (X,B)\dasharrow (X',B')$ between two natural completions, 
	which is not an isomorphism,  such that both $\varphi$ and $\varphi^{-1}$ 
	have at most one $\KK$-base-point 
	(and in particular there is no  infinitely near base-point). An \emph{elementary link} 
	is a link  which does not decompose into $\varphi=\varphi'\circ \varphi''$, 
	where $\varphi',\varphi''$ are links.
\end{defn}
\begin{rem}\label{Rem:ComposLinkAuto}
It follows from Definition~\ref{Def:Elmlinks} that if $\varphi\colon (X,B)\dasharrow (X',B')$ is an elementary link, then so is $\alpha\varphi\beta$, where $\alpha\in\Aut(X',B'), \beta\in \Aut(X,B)$. 
\end{rem}

The following lemma makes a precise description of our elementary links; this shows in particular that such maps are some of the elementary links which appear in the minimal model program (see \cite[Theorem of page 225]{Cor95} and~\cite[Definition~2.2, page 597]{Isk96}):
	
\begin{lem}\label{Lem:DescriptElLink}
	Any elementary link is of one of the following three types:
	\begin{enumerate}[$i)$]
 	\item link of type $\mathrm{I}$: a rational map 
	$(\PP^2,L_{\p^2})\dasharrow (\FF_1,E_1\cup L_{\mathbb{F}_1})$ 
	consisting of blowing-up one $\K$-rational point of $L_{\p^2}$.
  
  	\item link of type $\mathrm{II}$: rational maps 
	$(\FF_n,E_n\cup L_{\FF_n})\dasharrow (\FF_{m},E_m\cup L_{\FF_{m}})$ 
	given by the blow-up of a $\K$-rational point $p$ 
	of $L_{\FF_n}$ followed by the contraction of the strict transform of 
	$L_{\FF_n}$. Moreover, $m=n+1$ if  $p\in E_n$ and $m=n-1$ if $p\notin E_n$.
  
  	\item link of type $\mathrm{III}$: a morphism 
	$(\FF_1,E_1\cup L_{\mathbb{F}_1})\to (\PP^2,L_{\p^2})$ given 
	by the contraction of the curve $E_1$ onto a $\K$-rational point of $L_{\p^2}$.
	\end{enumerate}
	The inverse of a link of type $\mathrm{I}$, $\mathrm{II}$, $\mathrm{III}$ is a link of type $\mathrm{III}$, $\mathrm{II}$, $\mathrm{I}$ respectively.
\end{lem}
\begin{proof}
Let $\varphi\colon (X,B)\bir (X',B')$ be a birational map, {that} is an elementary link.

Suppose first that $\varphi$ is a morphism, which implies that $\varphi$ is the contraction of a $(-1)$-curve onto a $\K$-rational point. {Hence,} $X=\FF_1$, $X'=\p^2$ and we are in the third case. If $\varphi^{-1}$ is a morphism, we get symmetrically the first case.

In the remaining cases, both $\varphi$ and $\varphi^{-1}$ have exactly one base-point. These points are thus defined over~$\K$, and both maps contract one irreducible curve. The curves have thus either self-intersection $-1$ or self-intersection $0$, depending if they contain the base-points. In the first case, we get a link $(\FF_1,E_1\cup L_{\mathbb{F}_1})\dasharrow (\FF_1,E_1\cup L_{\mathbb{F}_1})$, which is not an elementary link since it factors through $\p^2$ as the composition of links of type $\mathrm{III}$ and $\mathrm{I}$. In the second case, we get a link of type $\mathrm{II}$ described above.
\end{proof}

\begin{prop}
	\lab{PROP:ExistenceDecomposition}
	Let $\varphi \colon (X, B) \bir (X', B')$ be a birational map between
	two natural completions of $\AA^2$. If $\varphi$ is not an isomorphism, then
	there exist $m \geq 1$ and 
	elementary links $\varphi_1, \ldots, \varphi_m$ such that
	$\varphi = \varphi_m \cdots \varphi_1$.
\end{prop}
\begin{rem}We call $\varphi = \varphi_m \cdots \varphi_1$ a \emph{decomposition into $m$ elementary links}. If $\varphi$ is an isomorphism, we sometimes say that it decomposes into $0$ elementary links. This is coherent with the fact that the composition of an elementary link with an {isomorphism} of natural 
completions is an elementary link (see Remark~\ref{Rem:ComposLinkAuto}).

Proposition~\ref{PROP:ExistenceDecomposition} implicitly follows from the work made in \cite{Lam02} or from \cite[Theorem 3.0.2]{BlDu}. {However}, the proof given here is direct.
\end{rem}
\begin{proof}
As all birational maps between projective smooth surfaces, $\varphi$ admits a minimal resolution, i.e.\;two birational morphisms $\epsilon\colon Y\to X$, $\eta\colon Y\to X'$ such that $\varphi=\eta\epsilon^{-1}$, and such that no $(-1)$-curve of $Y$ (not necessarily defined over~$\K$) is contracted by both $\epsilon$ and $\eta$. Moreover, each point blown-up by $\eta$ and $\epsilon$ belongs to the boundaries $B,B'$, as proper or infinitely near points.

We proceed by induction on the number of $\KK$-points blown-up by $\eta$ and $\epsilon$, corresponding to the number of base-points of $\varphi^{-1}$ and $\varphi$.
If $\eta$ is an isomorphism, $\varphi$ is only a sequence of blow-ups in the boundary of $B$. Because of the nature of $B$ and $B'$, this implies that $\varphi$ is either an isomorphism or a link of type $\mathrm{I}$, from $X=\p^2$ to $X'=\FF_1$. Similarly, $\varphi$ is an isomorphism or a link of type $\mathrm{III}$ if $\epsilon$ is an isomorphism.

Thus, we may assume that  $\eta$ (respectively $\epsilon$) contracts at least one $(-1)$-curve of $Y$, not contracted by $\epsilon$ (respectively $\eta$), and which is thus the strict transform of an irreducible curve $E\subset B$ (respectively $E'\subset B'$) of self-intersection $\ge -1$. If $E^2=-1$, we factorise $\varphi$ with a link of type $\mathrm{III}$ from $X=\FF_1$ to $\p^2$. We do the same with $\varphi^{-1}$ if $(E')^2=-1$.

Looking at the self-intersections of the curves on the boundaries $B$ and $B'$, the remaining case is: only one (-1)-curve of $Y$ is contracted by $\eta$ (respectively $\varepsilon$) and its image $E$ (respectively $E'$) under
$\varepsilon$ (respectively $\eta$) has self-intersection $\ge 0$.
This implies that $\varphi$ and $\varphi^{-1}$ have exactly one proper base-point. If $E^2=1$, we factorise $\varphi$ with a link of type $\mathrm{I}$ from $X=\p^2$ to $\FF_1$. Otherwise, $E^2=0$ and we factorise $\varphi$ with a link of type $\mathrm{II}$ from $X=\FF_n$ to $\FF_{n'}$ with $n'=n\pm 1$. It remains to see that $n'=0$ is impossible. Indeed, otherwise the $(-1)$-curve of $X$ would not pass through the unique proper base-point of $\varphi$ and would thus be sent by $\varphi$ onto a curve of self-intersection $\ge 0$, not contracted by~$\varphi^{-1}$.
\end{proof}
\begin{cor}
	Let $\varphi \colon (X, B) \bir (X', B')$ be a birational map between
	two natural completions of $\AA^2$. All $\KK$-base-points of $\varphi$ $($that belong to $X$ as proper or infinitely near points$)$ are defined over~$\K$.
\end{cor}
\begin{proof}
Follows from Proposition~\ref{PROP:ExistenceDecomposition} and from the fact that the base-points of any elementary link are defined over~$\K$ (Lemma~\ref{Lem:DescriptElLink}).
\end{proof}
Let us recall how the result on elementary links {gives} the proof of Jung - van der Kulk's Theorem (that we will not need in the sequel).
\begin{cor}
The group $\Aut(\A^2)$ is generated by the groups $$\Aff(\A^2)\mbox{ and }\J=\bigcup_{n\ge 1}\J_n.$$
\end{cor}
\begin{proof}
We view  $g\in \Aut(\A^2)$ as a birational map $(\PP^2,L_{\p^2})\dasharrow (\PP^2,L_{\p^2})$. If it is an isomorphism of pairs, then $g\in \Aff(\A^2)$. Otherwise, we use Proposition~\ref{PROP:ExistenceDecomposition} and write it as $g=\varphi_m \cdots \varphi_1$, where the $\varphi_i$ are elementary links. Denote by $r$ the {smallest} integer such that $\varphi_r$ {is} of type $\mathrm{III}$. Then, 
$\varphi_{r-1} \cdots \varphi_2\colon (\FF_1,E_1\cup L_{\FF_1})\dasharrow (\FF_1,E_1\cup L_{\FF_1})$ is a composition of links of type $\mathrm{II}$, which preserves the fibration $\FF_1\to \p^1$ and  the map $\varphi_{r}\cdots\varphi_1\colon (\PP^2,L_{\p^2})\dasharrow (\PP^2,L_{\p^2})$ sends thus the lines through the point $q_1\in L_{\p^2}$ blown-up  by $\varphi_1$ onto the lines through the point $q_2\in L_{\p^2}$ blown-up by $(\varphi_r)^{-1}$. Choosing $\alpha,\beta\in \Aut(\PP^2,L_{\p^2})$ that send respectively $q_1,q_2$ onto $(0:1:0)$, the map $\beta\varphi_{r}\cdots\varphi_1\alpha^{-1}$ preserves the lines through $(0:1:0)$. Hence, its restriction to $\A^2$ belongs to $\J$. Since the restriction of $\alpha,\beta$ belongs to $\Aff(\A^2)$, the map $\varphi_{r}\cdots\varphi_1$ belongs to the group generated by $\Aff(\A^2)$ and $\J$. Proceeding by induction on $m$, we get the result.
\end{proof}
For our purpose, we will need the decomposition into elementary links given by Proposition~\ref{PROP:ExistenceDecomposition}, but we will also need to have more precise information on the composition of links and their base-points, provided by Lemma~\ref{Lem:SmallReduction} and Proposition~\ref{PROP:ReducedDecomposition} below.

\begin{lem}\label{Lem:SmallReduction}
Let $\varphi\colon (X_0,B_0)\bir (X_1,B_1)$ and $\psi\colon (X_1,B_1)\bir (X_2,B_2)$ be two elementary links. The birational map $\psi \circ \varphi$ is an isomorphism if and only if one of the following occurs:
\begin{enumerate}[$i)$]
\item {$X_1=\FF_1$, $X_0=X_2=\p^2$ 
$($i.e.\;{$\psi$ and $\varphi^{-1}$ are both of type $\mathrm{III}$}$)$}
\item  {The maps} $\psi$ and $\varphi^{-1}$ are both of type $\mathrm{I}$ or both of type $\mathrm{II}$, and share the same base-point.
\end{enumerate}
Moreover, if $X_1=\FF_{n}$ and $X_0=X_2=\FF_{n+1}$, for some $n\ge 1$, the map $\psi\circ \varphi$ is always an isomorphism.
\end{lem}
\begin{proof}
If neither $\psi$ nor $\varphi^{-1}$ has a base-point, we have $X_1=\FF_1$ and $X_0=X_2=\p^2$, and both  $\psi$ and $\varphi^{-1}$ are contractions of the $(-1)$-curve of $\FF_1$. This implies that {$\psi \circ \varphi$} 
is an automorphism of $\p^2$.

If $\psi$ has a base-point, which is not a base-point of $\varphi^{-1}$, then 
{$\psi \circ\varphi$} has a base-point, and is thus not an isomorphism. 
The same holds exchanging the role of $\psi$ and $\varphi$.

The last case is when both $\psi$ and $\varphi^{-1}$ are links of type $\mathrm{I}$ or $\mathrm{II}$ {and have} the same base-point. The description of the links (Lemma~\ref{Lem:DescriptElLink}) implies that {$\psi \circ \varphi$} is an isomorphism. 
{It remains to see that this is always the case when $X_1=\FF_{n}$ and $X_0=X_2=\FF_{n+1}$. Indeed, this is true as the base-point of every elementary link from $(\FF_{n},E_n\cup L_{\mathbb{F}_n})$ to $(\FF_{n+1},E_{n+1}\cup L_{\mathbb{F}_{n+1}})$ is the intersection point of $E_n$ and $L_{\FF_n}$.}
\end{proof}

\begin{defn}
	A decomposition of a birational map $\varphi \colon (X, B) \bir (X', B')$ 
	into elementary links is called $\emph{reduced}$,
	if the composition of two consecutive elementary links is never an automorphism.
\end{defn}

\begin{prop}
	\lab{PROP:ReducedDecomposition}
	Let $\varphi \colon (X, B) \bir (X', B')$ be a birational map between
	two natural completions of $\AA^2$, and let 
	$\varphi = \varphi_m \cdots \varphi_1$ be a reduced decomposition
	of $\varphi$ into elementary links, with $m\ge 1$. Then, the following hold:
	\begin{enumerate}[$i)$]
	\item The number $l$ of base-points of $\varphi$ is equal to 
	the number of elementary links 
	$\varphi_i$ of type $\mathrm{I}$ or $\mathrm{II}$. If $l\ge 1$, then $\varphi$ has a unique
	proper base-point, which is equal to the base-point of $\varphi_1$ if $\varphi_1$ is of type $\mathrm{I}$ or $\mathrm{II}$, or to the preimage by $\varphi_1$ of the base-point of $\varphi_2$ if $\varphi_1$ is of type $\mathrm{III}$.
	\item If $\varphi=\psi_k \cdots \psi_1$ is 
	another decomposition into elementary links of $\varphi$, then 
	$k\ge m$.
	Moreover, if $k=m$ the decomposition is unique, up to isomorphisms of completions.
	\end{enumerate}
	
\end{prop}
\begin{proof}
We prove {$i)$} by induction on $m$, the case $m=1$ being obvious. 
Let $m \ge 2$.

Suppose that $\varphi_1$ is a link of type $\mathrm{III}$. It contracts $E_1\subset X$ onto a point which is not a base-point of $\varphi_2$, and thus not of $\varphi_m\cdots\varphi_2$, by induction hypothesis. In consequence, the number of base-points of $\varphi$ is equal to the one of $\varphi_m\cdots\varphi_2$, and $\varphi$ has a unique proper base point. This latter is the preimage by $\varphi_1$ of the {proper base-point} of $\varphi_2$.

Suppose now that $\varphi_2$ is a link of type $\mathrm{III}$, which implies that $\varphi_1$ is a link $\FF_2\dasharrow \FF_1$ of type $\mathrm{II}$, with one base-point $q$. If $m=2$, the result is clear. If $m\ge 3$, we use {the} induction hypothesis and see that the unique proper base-point of $\varphi_m\cdots\varphi_2$ is a point of $L_{\FF_1}\setminus E_1$, which is not a base-point of $(\varphi_1)^{-1}$. This implies that $\varphi$ has one more base-point as 
$\varphi_m\cdots\varphi_2$. Moreover, since the
proper base point of $\varphi_m\cdots \varphi_2$ is on $L_{\FF_1}$ 
and since $L_{\FF_1}$ is contracted by $(\varphi_1)^{-1}$ onto $q$, this implies that $q$ is the only proper base-point of $\varphi$.

To prove {$i)$}, it remains to study the case where $\varphi_1$ and $\varphi_2$ are links of type $\mathrm{I}$ or~$\mathrm{II}$. As in the previous case, $\varphi_m\cdots\varphi_2$ has a unique proper base-point (here equal to the one of $\varphi_2$), which is not a base-point of $(\varphi_1)^{-1}$ but belongs to the curve contracted by $(\varphi_1)^{-1}$ onto the base-point of $\varphi_1$. This again implies that $\varphi$ has one more base-point as $\varphi_m\cdots\varphi_2$ and that its unique proper base-point is the one of~$\varphi_1$. 

Now that {$i)$} is proved, let us show that it implies {$ii)$}. Suppose the existence of another decomposition, that we can assume to be reduced, of length $k\le m$: $\varphi=\varphi_m\cdots \varphi_1=\psi_k\cdots \psi_1$. If $\varphi$ has no base-point, both {decompositions} consist of one link of type $\mathrm{III}$ and the result is clear. Otherwise, the unique base-point of $\varphi$ determines the first link, so we have $\psi_1=\alpha\varphi_1$ for some isomorphism of {natural completions} 
$\alpha$. Proceeding by induction, we get $k=m$ and the unicity of the decomposition.
\end{proof}

Proposition~\ref{PROP:ReducedDecomposition} implies that the ``length" of a reduced decomposition of $\varphi$ only depends on $\varphi$. The following definition is thus natural.
\begin{defn}
	The number of elementary links in a reduced decomposition of $\varphi$
	is called the \emph{length of $\varphi$} and we denote it by $\len(\varphi)$.  
\end{defn}
\section{Birational maps preserving a curve and the proof of Theorem~\ref{Thm:Main}}
\label{Sec:ProofThm}
This section is devoted to the proof of Theorem~\ref{Thm:Main}, which is done using elementary links and by looking at the singularities of the curve obtained in the natural completions.
\begin{defn}
	If $\Gamma \subseteq X$ is a (closed) curve with no component in $B$
	and $\varphi \colon (X, B) \bir (X, B)$ is a birational map, 
	then we denote by $\varphi(\Gamma)$ the closure of  
	$\varphi(\Gamma \cap X \setminus B)$ in $X$ 
	and call it the \emph{image of $\Gamma$ under $\varphi$}.
	Moreover, we denote by $\Bir((X, B), \Gamma)$ the group of birational maps 
	$\varphi \colon (X, B) \bir (X, B)$ such that $\varphi(\Gamma) = \Gamma$.
\end{defn}
\begin{rem}
If $\Gamma\subset X$ is a curve having no component in $ B$, then $\Bir((X,B),\Gamma)=\{g\in \Aut(X\setminus B)\ |\ g (\Gamma\setminus B)=(\Gamma\setminus B)\}$. In particular, $\Bir((X,B),\Gamma)$ corresponds to the group $\Aut(\A^2,\Gamma\cap \A^2)$ of automorphisms of $X\setminus B=\A^2$ that preserve the closed curve $\Gamma\cap \A^2$ of $\A^2$.
\end{rem}

To study maps preserving a curve, we will study the singularities of the {curve} on the boundary.

\begin{defn}
	Let $\Gamma$ be a curve and let $p \in \Gamma$ be  a 
	$\KK$-point. We define the 
	\emph{height $\h_\Gamma(p)$ of $\Gamma$ at $p$} inductively, as follows:
	\begin{enumerate}[$i)$]
		\item If $\Gamma$ is smooth in $p$, then we define 
		$\h_\Gamma(p) := 0$.
		\item Otherwise, let $\pi \colon \tilde{\Gamma} \to \Gamma$ 
		be the blow-up of $\Gamma$ 
		in $p$ and let $p_1, \ldots, p_n$ be the  
		points of $\tilde{\Gamma}$ with $\pi(p_i) = p$. We define
		$\h_\Gamma(p) := \max_i \{ \, \h_{\tilde{\Gamma}}(p_i) + 1 \, \}$.
	\end{enumerate}
\end{defn}

Recall that for any curves $\Gamma_1,\Gamma_2$ having no common component on a smooth projective surface, the intersection number $\Gamma_1\cdot \Gamma_2$ is non-negative, and corresponds to the sum of the intersection numbers $(\Gamma_1\cdot \Gamma_2)_p$, where $p$ runs over all $\KK$-points of $\Gamma_1\cap \Gamma_2$. The intersection number $(\Gamma_1\cdot \Gamma_2)_p$ satisfies $(\Gamma_1\cdot \Gamma_2)_p\ge m_p(\Gamma_1)\cdot m_p(\Gamma_2)$, where $m_p(\Gamma_i)$ is the multiplicity of $\Gamma_i$ at $p$, and equality holds if and only if the tangent cones of $\Gamma_1$ and $\Gamma_2$ at $p$ are distinct.\\

The next proposition is the key ingredient in the proof of Theorem~\ref{Thm:Main}.

\begin{prop}
	\lab{PROP:Key}
	Let $\varphi \colon (X, B) \bir (X', B')$ be a birational map 
	between natural completions of $\AA^2$, admitting a reduced
	decomposition $\varphi=\varphi_m\cdots\varphi_1$ into 
	elementary links, with $m \ge 1$.
	Let $\Gamma\subset X$  be a curve having no component in  $B$ {and}
	let $\Gamma'=\varphi(\Gamma)\subset X'$ be its image under $\varphi$. 
	Suppose that one of the following holds:
	\begin{enumerate}[$a)$]
	\item The link $\varphi_1$ is of type $\mathrm{III}$ 
	$($from $\mathbb{F}_1$ to $\p^2)$
	and $(E_1 \cdot \Gamma)_{p_0} > 1$, where $p_0$ is defined by
	{$\{p_0\} =E_1\cap L_{\FF_1} =: A$}.
	
	\item The link $\varphi_1$ is of type $\mathrm{I}$ or $\mathrm{II}$, and 
	the finite set $A\subset L_X$ of $\KK$-points lying on $\Gamma$ 
	which are not 
	$($proper$)$ base-points of $\varphi$ satisfies
	$\sum_{p\in A} (L_X \cdot \Gamma)_p>1$.
	\end{enumerate}	
	Then, there exists a $\KK$-point {$q\in L_{X'}$} such that
	\[
		\h_{\Gamma'}(q)\ge l+\max_{p\in A} \h_{\Gamma}(p) \, ,
	\]
	where $l$ is the number of base-points 
	of $\varphi^{-1}$. If $l\ge 1$, then $q$ can be
	chosen as the unique proper base-point of $\varphi^{-1}$, which is a $\K$-point.
\end{prop}
%

\begin{proof}
	We proceed by induction on the number $m$ of elementary links in the 
	reduced decomposition of $\varphi$. 
	We distinguish the following cases, depending on the nature of the first link $\varphi_1$.
	
	$i)$ If $\varphi_1$ is a link of type $\mathrm{III}$, 
	it contracts the curve $E_1 \subseteq \FF_1$
	onto a $\K$-point $q\in \p^2$. By assumption, we have 
	$(E_1 \cdot \Gamma)_{p_0} > 1$. 
	The multiplicity of $\Gamma_1=\varphi_1(\Gamma)$ at  $q$ is 
	equal to $E_1\cdot \Gamma>1$, 
	which implies that $\Gamma_1$ is singular at
	$q$ and that $\h_{\Gamma_1}(q)\ge \h_{\Gamma}(p_0)+1$.
	
	If $m = 1$, we have $l=1$ and 
	$\h_{\Gamma_1}(q)\ge \h_{\Gamma}(\tr{p_0})+1= 
	l+\max_{p\in A} \h_{\Gamma}(p)$
	 (because $A=\{p_0\}$), so we are done. So, assume $m > 1$. The point 
	$q$ is not a base point of $\varphi_2$, since otherwise
	$\varphi_2 \circ \varphi_1$ would be an automorphism of $\mathbb{F}_1$.
	Moreover, 
	$(L_{\PP^2} \cdot \Gamma_1)_q > 1$ because $\Gamma_1$ is 
	singular at $q$.
	The claim follows by applying the induction hypothesis to 
	$\varphi_m \cdots \varphi_2$, as $q$ is not a base-point of
	$\varphi_m \cdots \varphi_2$.
	
	$ii)$ Suppose that $\varphi_1$ is a link of type $\mathrm{II}$, 
	i.e.\;$\varphi_1 \colon \FF_{n'} \bir \FF_{n}$ where $n'=n \pm 1$. 
	By definition there
	exist blow-ups $\varepsilon \colon S \to \FF_{n'}$ and 
	$\eta \colon S \to \FF_{n}$ of  $\K$-points $q'\in L_{\FF_{n'}}$ and 
	$q\in L_{\FF_{n}}$ respectively, such that 
	$\varphi_1 = \eta \circ \varepsilon^{-1}$. Denote by $\tilde{\Gamma} \subseteq S$
	the strict transform of $\Gamma$ under $\varepsilon^{-1}$, and by $E_q\subset S$ 
	the irreducible curve contracted by $\eta$ onto $q$, which is the strict transform
	of $L_{\FF_{n'}}$ under $\varepsilon^{-1}$. 
	The map $\varepsilon^{-1}$ is an isomorphism in a neighbourhood
	of every point of $A\subset L_{\FF_{n'}}$.
	By hypothesis, one has $\sum_{p\in A} (L_{\FF_{n'}} \cdot \Gamma)_p>1$, which implies that 
	$\tilde{\Gamma}\cdot E_q>1$ on $S$. The curve $\Gamma_1=\eta(\tilde{\Gamma})\subset \FF_n$ is thus singular at $q$, 
	and the height of $\Gamma_1$ at $q$ 
	is at least equal to $1+\max_{p\in A} \h_{\Gamma}(p)$.
	If $m = 1$, then we are done. If $m > 1$, 
	$q$ is not a base point of $\varphi_2$, since otherwise 
	$\varphi_2 \circ \varphi_1$ would be an isomorphism. Moreover, $q$ is a singular point of $\Gamma_1$ which 
	belongs to $L_{\FF_n}$, and if $\varphi_2$ is of type $\mathrm{III}$, then $n=1$, $n'=2$, so $q$ 
	is the intersection point of $E_1$ and $L_{\FF_1}$. In any case, the point $q$ belongs to the set ``$A$" associated
	to $\varphi_m \cdots \varphi_2$, so the result follows by induction.
	
	$iii)$ The last case is when $\varphi_1$ is of type  $\mathrm{I}$, i.e.\;when
	it is a map $\varphi_1\colon \PP^2 \bir \FF_1$.
	 If $m = 1$, {then} $l=0$ and the result is obvious, by choosing $q$ as the image of a point of $A$ with a maximal height.
	  If $m>1$, the link $\varphi_2$ is a map $\FF_1\dasharrow \FF_2$ centered at the intersection point of 
	  $E_1$ and $L_{\FF_1}$. In particular, the map $\varphi_1$ induces an isomorphism at a neighbourhood of any point
	  of $A$, sending $L_{\p^2}$ onto $L_{\FF_1}$, and sends the set $A$ onto the set ``$A$" associated 
	  to $\varphi_m \cdots \varphi_2$. The result follows by induction hypothesis.
\end{proof}

The following corollary is a direct consequence of  Proposition~\ref{PROP:Key}.

\begin{cor}
	\lab{COR:Key}
	Let $\varphi \in \Bir((X, B), \Gamma)\setminus \Aut(X,B)$ and let 
	$\varphi = \varphi_m \cdots \varphi_1$ be its reduced decomposition
	into elementary links.
	\tr{We assume that $\varphi_1$ is of type $\mathrm{I}$ or $\mathrm{II}$, 	
	denote by $q$ the base-point of $\varphi_1$ $($which is the base-point of 
	$\varphi$$)$ and denote by $A\subset L_X$ the finite set of 
	$\KK$-points lying on $\Gamma$ which are not 
	$($proper$)$ base-points of $\varphi$. Then one of the following holds:}

	\tr{$1)$ The set $A$ \tr{is empty or} it consists of one point $p$ 
	which satisfies 
	$(L_X \cdot \Gamma)_p = 1$.}
	
	\tr{$2)$ The set $A$ is non-empty and 
	the height of $\Gamma$ at $q$ is bigger than
	the height of $\Gamma$ at every  point of $A$
	$($and thus $\Gamma$ is singular at $q$$)$.}
\end{cor}

\begin{prop}
	\lab{PROP:Main1}
	Let $(X, B)$ be a natural completion of $\AA^2$ and let $\Gamma$ be a curve in $X$ having no component in $B$. Then, 
	there exists a natural completion $(X', B')$
	of $\AA^2$ and a birational map $\varphi \colon (X, B) \bir (X', B')$ 
	such that one of the following holds:
	
	$1)$ $\Bir((X',B'),\varphi(\Gamma)) \subseteq \Aut(X', B')$;
	
	$2)$  The curve $\varphi(\Gamma)$ intersects transversally $L_{X'}$, i.e.\;$(\varphi(\Gamma)\cdot L_{X'})_p\le 1$ for any $p\in X'(\KK)$.
\end{prop}

\begin{proof}
	{Define} $(X_0,B_0)=(X,B)$ and $\varphi_0 = \id \in \Aut(X_0, B_0)$.
	We construct inductively a sequence of elementary links
	\[
		\label{EQ:Sequence}
		\tag{$\ast$}
		(X_0, B_0) \stackrel{\varphi_1}{\bir} 
		(X_1, B_1) \stackrel{\varphi_2}{\bir}
		(X_2, B_2) \stackrel{\varphi_3}{\bir} \ldots \, .
	\]
	Let $i \geq 0$ and assume
	that $\varphi_0, \ldots, \varphi_i$ are already constructed. We 
	write $\Gamma_i=\varphi_i\cdots\varphi_1(\Gamma)$, $G_i=(\Bir(X_i,B_i),\Gamma_i)$ 
	and define  $\varphi_{i+1}$ in the following way:
	
	\begin{enumerate}[$i)$]
	\item If  $G_i \subseteq \Aut(X_i, B_i)$ or 
	$\Gamma_i$ intersects transversally $L_{X_i}$, we define 
	$(X_{i+1}, B_{i+1}) = (X_i, B_i)$ and $\varphi_{i+1}= \id$.
	
	\item If $i)$ {does not} hold
	and the reduced decomposition into elementary links
	of every $g \in G_i \setminus \Aut(X_i, B_i)$ starts with an elementary link 
	$\tau_0\colon\FF_1 \to \PP^2$ of type III, then we define 
	 $\varphi_{i+1}=\tau_0$, and thus 
	 set $(X_{i+1},B_{i+1})=(\p^2,L_{\p^2})$.
	
	\item If  $i)$ and $ii)$ do not occur, there exists  
	$g \in G_i \setminus \Aut(X_i, B_i)$ such that
	the reduced decomposition into elementary links of $g$ 
	starts with an elementary link $\tau_1$, not of type III.
	In this case we define $\varphi_{i+1}=\tau_1$, and $(X_{i+1},B_{i+1})$ 
	as the target of $\tau_1$.
	\end{enumerate}
	
	Now, we claim that the sequence \eqref{EQ:Sequence} satisfies the following
	two properties.
	\begin{enumerate}[$a)$]
	\item If $\varphi_{i+1}$ is not an automorphism, then we have 
	for all $g \in G_{i}$ 
	\[
		\len(\varphi_{i+1} g \varphi_{i+1}^{-1}) < \len(g) 
		\quad \textrm{or} \quad
		\len(\varphi_{i+1} g \varphi_{i+1}^{-1}) = \len(g) = 0 \, .
		\]
	
	\item The sequence \eqref{EQ:Sequence} is stationary after
	finitely many steps, i.e.\;$\varphi_i$ is an automorphism
	for $i$ large enough.
	\end{enumerate}
	
	Property $a)$ will serve to show Property $b)$, which
	directly implies the result. It thus remains to prove these two properties.
	
	\smallskip
	
	{\it Proof of Property $a)$}:
	
	In case $i)$, there is nothing to prove.
	
	If we are in case $ii)$, then for all 
	$g \in G_i \setminus \Aut(X_i, B_i)$ the reduced decomposition 
	into elementary links starts with $\tau_0\colon \FF_1 \to \PP^2$
	and ends with $(\tau_0)^{-1}$, since
	$G_i$ is a group. The conjugation of $G_i$ by $\tau_0$ decreases the 
	length of any element of $G_i\setminus \Aut(X_i, B_i)$ by two. 
	Moreover, $(X_i,B_i)=(\FF_1,E_1\cup L_{\FF_1})$, so $\Aut(X_i,B_i)$
	preserves the curve $E_1$ contracted by $\tau_0$, 
	which implies that 
	$\tau_0 \Aut(X_i, B_i) (\tau_0)^{-1}\subset \Aut(\p^2,L_{\p^2})=
	\Aut(X_{i+1},B_{i+1})$. It follows that a) is satisfied.
			
	Assume that we are in case $iii)$. As we are not in case $i)$
	there exists $p_0 \in X_i$ such that $(\Gamma_i \cdot L_{X_i})_{p_0} > 1$. 
	Moreover, by definition of case $iii)$, there exists
	$g \in G_i \setminus \Aut(X_i, B_i)$ such that the reduced decomposition
	into elementary links starts with 
	$\varphi_{i+1}\colon (X_i,B_i)\dasharrow (X_{i+1},B_{i+1})$, 
	and $\varphi_{i+1}$ has a base-point $p$ since it is not of type III.
	Applying Corollary~\ref{COR:Key} to $g$, we see that 	
	$(\Gamma_i \cdot L_{X_i})_{p} > 1$, 
	\tr{ and that either the height 
	of $\Gamma_i$ at 
	$p$ is bigger than the height of $\Gamma_i$ at 
	every other point of 
	$\Gamma_i \cap L_{X_i}$ and $\Gamma_i$ is singular at $p$, 
	or that all points
	in $\Gamma_i \cap L_{X_i} \setminus \{ p \}$ intersect transversally 
	$L_{X_i}$.}
	This implies that any element $h \in \Aut(X_i,B_i) \cap G_i$ fixes $p$, 
	which implies $\varphi_{i+1} h \varphi_{i+1}^{-1} \in \Aut(X_{i+1},B_{i+1})$.
	We can then  apply Corollary~\ref{COR:Key} to any element 
	$h\in G_i \setminus \Aut(X_i, B_i)$ \tr{that does not start with a link of 
	type $\mathrm{III}$}, and see that $p$ is the unique proper 
	base-point of $h$, so the reduced decomposition of $h$ starts with 
	$\varphi_{i+1}$. \tr{It remains to observe that no element 
	$h\in G_i \setminus \Aut(X_i, B_i)$ starts with a link 
	$\tau$ of type $\mathrm{III}$. Indeed, otherwise we 
	would have $X_i=\FF_1$ and $\varphi_{i+1}$ is an elementary 
	link from $\FF_1$ to $\FF_2$, so $p$ is the intersection point 
	$E_1\cap L_{\FF_1}$.  Since $(\Gamma_i \cdot L_{\FF_1})_p > 1$, 
	it follows that 
	$(\tau(\Gamma_i) \cdot L_{\p^2})_{\tau(p)}>
	(\Gamma_i \cdot L_{\FF_1})_p > 1$. 
	We can then apply Proposition~\ref{PROP:Key} to 
	$h\tau^{-1}\colon (\p^2,L_{\p^2})\dasharrow (\FF_1,E_1\cup L_{\FF_1})$, 
	and find some other point $p'\in L_{\FF_1}$ where $\Gamma_i$ has height (strictly)
	bigger than $p$, or such that $\Gamma_i$ is smooth at $p$ and $p'$ and  
	$(\Gamma_i \cdot L_{\FF_1})_{p'}>(\Gamma_i \cdot L_{\FF_1})_{p}$, 
	a contradiction.} This yields $a)$.		
		
	\smallskip
		
	{\it Proof of Property $b)$}:
	
	{Property $a)$ implies that for each $i$ such that $\varphi_1,\dots,\varphi_i$ are not isomorphisms, the product $\varphi_i \cdots \varphi_1$ is
	a reduced decomposition into elementary links}.
	For any $i\ge 0$ such that neither
	$\varphi_{i+1}$ nor $\varphi_{i+2}$ is an isomorphism, 
	we define $p_i\in L_{X_i}\subset B_i\subset X_i$ as the unique
	proper base-point of $\varphi_{i+2} \circ \varphi_{i+1}$.
	
	Under the assumption that neither of the three maps 
	$\varphi_{i+1}, \varphi_{i+2}$ and $\varphi_{i+3}$ is an isomorphism
 	we show that 
	$(\Gamma_i\cdot L_{X_i})_{p_i}\ge 
	(\Gamma_{i+1}\cdot L_{X_{i+1}})_{p_{i+1}}$ and 
	give some consequences when equality holds:
	\begin{itemize}
		\item {\it $\varphi_{i+1}$ is a link of type $\mathrm{II}$}:
		we can write $\varphi_{i+1} = \eta \circ \pi^{-1}$, where 
		$\pi\colon S\to X_i$ is the blow-up of $p_i$
		and $\eta\colon S\to X_{i+1}$ is the blow-up of a point $q\not=p_{i+1}$.
		Denote by $\widetilde{\Gamma_i},\widetilde{L_{X_i}}\subset S$ 
		the strict transforms of $\Gamma_i$ and $L_{X_i}$, and let  
		$E_{p_i}$ be the exceptional divisor of $\pi$.
		It follows that {$\eta$ contracts $\widetilde{L_{X_i}}$}. 
		Note that $\eta^{-1}$ restricts to an 
		isomorphism in a neighbourhood of
		 $p_{i+1}$, which sends $L_{X_{i+1}}$ onto $E_{p_i}$. 
		 This yields the following estimate
		\begin{eqnarray*}
			(\Gamma_i \cdot L_{X_i})_{p_i} 
			&\geq& m_{p_i}(\Gamma_i) \\
			&=& \widetilde{\Gamma_i} \cdot E_{p_i} \\
			&\geq& 
			(\widetilde{\Gamma_i} \cdot E_{p_i})_{\eta^{-1}(p_{i+1})} \\
			&=& (\Gamma_{i+1} \cdot L_{X_{i+1}})_{p_{i+1}}
		\end{eqnarray*}
		where $m_{p_i}(\Gamma_i)$ denotes the multiplicty of $\Gamma_i$
		in $p_i$. 
		Moreover, if 
		$(\Gamma_i \cdot L_{X_i})_{p_i} = m_{p_i}(\Gamma_i) =
		(\Gamma_{i+1} \cdot L_{X_{i+1}})_{p_{i+1}} > 1$, then
		$\h_{\Gamma_i}(p_i) > \h_{\Gamma_{i+1}}(p_{i+1})$.
					
		\item {\it $\varphi_{i+1}$ is a link of type $\mathrm{I}$}:
		Remark that $(\varphi_{i+1})^{-1}\colon \FF_1\to\p^2$ is the blow-up of
		$p_i\in \p^2$ and that $\Gamma_{i+1}$ and
		$L_{\FF_1}$ are the strict transforms of $\Gamma_i$ and $L_{\PP^2}$ 
		under $(\varphi_{i+1})^{-1}$ respectively. It follows that
		$(\Gamma_i \cdot L_{\PP^2})_{p_i} \ge 
		(\Gamma_{i+1} \cdot L_{\FF_1})_{p_{i+1}} \, .$ {Moreover, the equality implies that $(\Gamma_i \cdot L_{\PP^2})_{p_i} = 
		(\Gamma_{i+1} \cdot L_{\FF_1})_{p_{i+1}}= 0$.}
				
		\item {\it $\varphi_{i+1}$ is a link of type $\mathrm{III}$}: 
		Since $\varphi_{i+1}$
		is an isomorphism in a neighbourhood of 
		{$p_i$ and $p_{i+1} = \varphi_{i+1}(p_i)$}
		it follows that
		$(\Gamma_i \cdot L_{\FF_1})_{p_i} = 
		(\Gamma_{i+1} \cdot L_{\PP^2})_{p_{i+1}}$ and
		$\h_{\Gamma_i}(p_i) = \h_{\Gamma_{i+1}}(p_{i+1})$.
	\end{itemize}
	
	Now, assume towards a contradiction, that the sequence
	\eqref{EQ:Sequence} is never stationary, i.e.\;$\varphi_{i}$ is not 
	an isomorphism for all $i \ge 1$.
	According to this case-by-case-analysis, we see that
	$(\Gamma_i \cdot L_{X_i})_{p_i}$ is a decreasing sequence in $i$
	and {for every $r > 1$ there are} only finitely many $i$ with
	$(\Gamma_i \cdot L_{X_i})_{p_i} = r$. Thus, there exists $I$
	such that for all $i \geq I$ we have $(\Gamma_i \cdot L_{X_i})_{p_i} {\leq 1}$.
	{As the sequence \eqref{EQ:Sequence} is not stationary, 
	we have a {$\varphi_{i+1}$} with 
	$i\ge I$ which is a link of type $\mathrm{I}$ or $\mathrm{II}$. 
	The point $p_i$ is thus the base-point of {$\varphi_{i+1}$} and 
	satisfies $(\Gamma_i \cdot L_{X_i})_{p_i} {\leq 1}$. 
	By Corollary~\ref{COR:Key} applied to an \tr{element} of $G_i\setminus \Aut(X_i,B_i)$ that starts with $\varphi_{i+1}$, this 
	implies that $(\Gamma_i \cdot L_{X_i})_q = 1$ for all $q \in L_{X_i} \cap \Gamma_i$. 
	Hence, we are in case $i)$, a contradiction.}
\end{proof}

\begin{lem}\label{LEM:Conic}
Let $(X,B)=(\p^2,L_{\p^2})$, and let $\Gamma$ be a conic $($as always, reduced but not necessarily irreducible$)$ in $X=\p^2$, intersecting $L_{\p^2}$ in two distinct points $($not necessarily defined over $\K)$.

Then,  $\Bir((X,B),\Gamma) \subseteq \Aut(X, B)$.
\end{lem}
\begin{proof}
Suppose, to the contrary, the existence of $g\in \Bir((X,B),\Gamma)$, which is not an automorphism of $\p^2$.

Applying Proposition~\ref{PROP:Key}, one of the two points of $\Gamma\cap L_X$ is a base-point of~$g$ (and in particular it is defined over~$\K$). By Proposition~\ref{PROP:ExistenceDecomposition} and 
	Lemma~\ref{Lem:SmallReduction}, there exists an integer $n\ge 2$ such that the reduced decomposition of $g$ into elementary links starts with 
$\varphi_{2n} \varphi_{2n-1} \cdots \varphi_1$, where $\varphi_1\colon \p^2\bir\FF_1$,  $\varphi_{2n}\colon \FF_1\to \p^2$, $\varphi_i\colon \FF_{i-1}\bir\FF_{i}$ and $\varphi_{i+n-1}\colon \FF_{n-i+2}\bir \FF_{n-i+1}$ for $i=2,\dots,n$. 
The curve $\varphi_1(\Gamma)$ intersects transversally $E_1$ in one point and $L_{\FF_1}$ in one point, and does not pass through the intersection point of $E_1$ and $L_{\FF_1}$, blown-up by $\varphi_2$. This implies that the same holds for $\varphi_2\varphi_1(\Gamma)$, with  $E_2$ and $L_{\FF_2}$, and that $\varphi_2\varphi_1(\Gamma)$ passes through the point blown-up by $(\varphi_2)^{-1}$. Proceeding by induction, the curve $\varphi_n\cdots\varphi_1(\Gamma)\subset \mathbb{F}_n$ intersects  transversally $E_n$ in one point and 
$L_{\FF_n}$ in one point, and 
{passes through the point blown-up by $(\varphi_n)^{-1}$}. In consequence, it 
does not pass through the point blown-up by $\varphi_{n+1}$, 
which implies that $\varphi_{n+1}\varphi_n\cdots\varphi_1(\Gamma)\subset \mathbb{F}_{n-1}$ intersects $E_{n-1} \cup L_{\FF_{n-1}}$ in two distinct points, both being on $E_{n-1}$. Proceeding by induction, the curve $\varphi_{2n-1}\cdots\varphi_n\cdots\varphi_1(\Gamma)\subset \mathbb{F}_{1}$ intersects $E_{1} \cup L_{\FF_{1}}$ in two distinct points, both on $E_1$.  The curve $\varphi_{2n}\cdots\varphi_1(\Gamma)\subset \p^2$ intersects $L_{\p^2}$ in one point $q$, and is singular at this point, with two branches. The remaining part of the decomposition of $g$ is equal to $\varphi_m\cdots\varphi_{2n+1}$, and its unique proper base-point is different from $q$. Proposition~\ref{PROP:Key} implies that $g(\Gamma)$ is singular at a point of $B=L_{\p^2}$, which is a contradiction.
\end{proof}
\begin{prop}
	\lab{PROP:Main2}
	Let $(X, B)$ be a natural completion of $\AA^2$ and let $\Gamma$
	be a curve in $X$.
	If $\Gamma$ intersects transversally $L_X$, then 
	there exists a natural completion 
	$(X', B')$
	of $\AA^2$ and a birational map $\varphi \colon (X, B) \bir (X', B')$ 
	such that one of the following holds:
	
	$i)$ $\Bir((X',B'),\varphi(\Gamma)) \subseteq \Aut(X', B')$;
	
	$ii)$ $X'=\p^2$ and $\varphi(\Gamma)\subset \p^2$ is defined by a polynomial in $\K[x]$.
\end{prop}

\begin{proof}
We can assume the existence of  $g\in \Bir((X,B),\Gamma)\setminus\Aut(X, B)$, admitting a reduced decomposition $g=\varphi_m\cdots \varphi_1$ into elementary links.

	{$a)$} Suppose that $X = \PP^2$. By Corollary~\ref{COR:Key}, the curve $\Gamma$ intersects $L_X$ in at most $2$ points, hence $\Gamma$ is a line or a conic. If $\Gamma$ is a line the equation can be chosen to be $x=0$, and if $\Gamma$ is a conic, the result follows from Lemma~\ref{LEM:Conic}.


	{$b)$} Suppose that {$X\not\simeq\p^2$} but that there exists $1\le i\le m$ such that $\varphi_i$ is  a link $\varphi_i \colon \FF_1\to \p^2$ of type $\mathrm{III}$, contracting $E_1$ onto a point $q\in L_{\p^2}$. If $\Gamma'=\varphi_i\cdots \varphi_1(\Gamma)$ intersects transversally $L_{\p^2}$, we conclude by applying case {$a)$} to $\Gamma'$. There exists thus a point $p\in L_{\p^2}$ with $(\Gamma'\cdot L_{\p^2})_p>1$. We claim that $p$ can be chosen to be equal to $q$. If $i=1$, this is true because $\Gamma'=\varphi_1(\Gamma)$ and $\Gamma$ intersects transversally $L_{\FF_1}$, which is the strict transform of $L_{\p^2}$. If $i>1$, the claim follows from Proposition~\ref{PROP:Key}, applied to 
$(\varphi_1)^{-1}\cdots(\varphi_i)^{-1}\colon (\p^2,\tr{L_{\p^2}})\dasharrow (X,B)$.
 Since $\varphi_{i+1}$ is a link of type $\mathrm{I}$ and $q$ is not a base-point of $\varphi_{i+1}$, $\varphi_{i+1}(\Gamma')$ does not intersect transversally $L_{\FF_1}$. This implies that $m\ge i+2$, and that $(\varphi_m\cdots\varphi_{i+1})^{-1}$ has at least one base-point. Applying Proposition~\ref{PROP:Key} to $\varphi_m\cdots \varphi_{i+1}$, the curve $\Gamma=\varphi_m\cdots\varphi_{i+1}(\Gamma')$ has a singular point at the proper base-point of $(\varphi_m\cdots\varphi_{i+1})^{-1}$. Since this one lies on $L_{X}$, we get a contradiction.
	
	{$c)$} We can now assume that $X=\mathbb{F}_n$ for some $n\ge 1$ and that all $\varphi_i$ are links of type $\mathrm{II}$. By Corollary~\ref{COR:Key}, $L_{\FF_n}\cap \Gamma$ contains at most $2$ points and if it contains $2$, then one of these is a base-point of $\varphi_1$.  
	If one of the points of $L_{\FF_n}\cap \Gamma$ is the intersection point $p_n$ of  $L_{\FF_n}$ and $E_n$, we perform an elementary link $\psi\colon \FF_n\bir \FF_{n+1}$. Because $L_{\FF_n}\cap \Gamma$ contains at most $2$ points, the curve $\psi(\Gamma)$ intersects transversally $L_{\FF_{n+1}}$, in at most $2$ points. Moreover, 
writing $\{p_{n+1}\}=L_{\FF_{n+1}}\cap E_{n+1}$, we obtain $(\psi(\Gamma)\cdot E_{n+1})_{p_{n+1}}=(\Gamma\cdot E_n)_{p_n}-1$. Performing a sequence of elementary links if needed, we reduce to the case where $L_{\FF_n}\cap \Gamma$  contains at most $2$ points and that none of them belong to $E_n$.

If $L_{\FF_n}\cap \Gamma$ is empty, then $\Gamma$ is contained in a finite set of fibres of $\FF_n \to \PP^1$. Going from $\FF_n$ to $\FF_1$ and then to $\p^2$, we send the fibres onto lines of the form $x=a$ where $a\in \KK$, and see that $L_{\FF_n}$ has equation in $\K[x]$.
	
	It remains to see that it is not possible, in the case where $L_{\FF_n}\cap \Gamma$  contains $1$ or $2$ points that belong to $L_{\FF_n}\setminus E_n$, to have $g\in \Bir((X,B),\Gamma)\setminus\Aut(X, B)$ having a reduced decomposition consisting only of links of type $\mathrm{II}$. Such an element has a decomposition $g=\varphi_{2k}\cdots \varphi_1$, where $k\ge 1$, $\varphi_i$ is a link $\FF_{n+i-1}\dasharrow \FF_{n+i}$ for $i=1,\dots,k$, $\varphi_i$ is a link $\FF_{n+2k-i+1}\bir \FF_{n+2k-i}$ for $i=k+1,\dots,2k$, and where $(\varphi_k)^{-1}$ and $\varphi_{k+1}$ do not have the same base-point (see Lemma~\ref{Lem:SmallReduction}). In particular, the base-point of $\varphi_1$ is the intersection point of $E_n$ and $L_{\FF_n}$, and does not lie on $\Gamma$, so $\varphi_1(\Gamma)$ passes through the base-point of $(\varphi_1)^{-1}$, which is not a base-point of $\varphi_2$. By induction, we deduce that $(\varphi_i\cdots \varphi_1)(\Gamma)$ passes through the base-point of $(\varphi_i)^{-1}$, for $i=1,\dots,2k$. In the case $i=2k$, this implies that $g(\Gamma)$ passes through the base-point of $(\varphi_{2k})^{-1}$, i.e.\;through the intersection point of $E_n$ and $L_{\FF_n}$, contradicting the fact that $g(\Gamma)=\Gamma$.
\end{proof}

\begin{cor}
	\lab{COR:Main}
	Let $(X, B)$ be a natural completion of $\AA^2$ and let $\Gamma$ be a curve in $X$ having no component in  $B$. Then, 
	there exists a natural completion $(X', B')$
	of $\AA^2$ and a birational map $\varphi \colon (X, B) \bir (X', B')$ 
	such that one of the following holds:
	
	$1)$ $\Bir((X',B'),\varphi(\Gamma)) \subseteq \Aut(X', B')$;
	
	$2)$  $X'=\p^2$ and $\varphi(\Gamma)\subset \p^2$ is defined by a polynomial in $\K[x]$.
\end{cor}

\begin{proof}
	Follows directly from Propositions~\ref{PROP:Main1} and~\ref{PROP:Main2}.
\end{proof}

Theorem~\ref{Thm:Main} is now a direct consequence of Corollary~\ref{COR:Main}:

\begin{proof}[Proof of Theorem~$\ref{Thm:Main}$]
In the case where the equation of $\Gamma$ is in $\K[x]$, i.e.\;when $\Gamma$ is a fence, the explicit description of $\Aut(\A^2,\Gamma)$ is an easy calculation. If $\Gamma$ is not equivalent to a fence by an automorphism of $\A^2$, Corollary~\ref{COR:Main} and Lemma~\ref{LEM:AutoNatCompl} imply that we can conjugate $\Aut(\A^2,\Gamma)$ to a subgroup of $\Aff(\A^2)$ or $\J_n$ for some $n$. This implies that $\Aut(\A^2,\Gamma)$ is an algebraic group. Moreover, we obtain a morphism of algebraic groups $\Aut(\A^2,\Gamma)\to \Aut(\Gamma)$. It remains to  observe that curves fixed pointwise by elements of $\J_n$ and $\Aff(\A^2)$ are 
{fences in a suitable coordinate system of $\AA^2$.}
\end{proof}

%
%
%
%
%
\subsection{Generalisation to other subsets}
\label{GeneralizationsOtherSubsets.subsec}

\begin{defn}
If $\Delta\subset \A^2(\KK)$ is any subset, we denote by $\Aut(\A^2,\Delta)$ the group of automorphisms of $\A^2$ that leave the set $\Delta$ invariant, and denote by $\Aut_F(\A^2,\Delta)$ the group of automorphisms of $\A^2$ fixing 
any element of $\Delta$.
\end{defn}

By Definition, $\Aut_F(\A^2,\Delta)$ is always a normal subgroup of $\Aut(\A^2,\Delta)$. Moreover, if $\Aut(\A^2,\Delta)$ is an algebraic group, then $\Aut_F(\A^2,\Delta)$ is a an algebraic subgroup. Theorem~\ref{Thm:Main} implies the following result:

\begin{prop}
Let $\Delta\subset \A^2(\KK)$ be a subset. Applying an element $\varphi\in \Aut(\A^2)$, {one of } the following holds:

\begin{enumerate}[$i)$]
\item
The set $\Delta$ is contained in a fence given by $f(x)=0$ where $f\in \K[x]$, and the group $\Aut_F(\A^2,\Delta)$ is not algebraic: it contains the group 
$$\{(x,y)\mapsto (x,y+p(x)f(x))\mid p\in \K[x]\}\simeq \K[x].$$
\item
The group $\Aut(\A^2,\Delta)$ is equal to $\{g\in \Aff(\A^2) \mid g(\Delta)=\Delta\}$, and is thus an algebraic subgroup of $\Aff(\A^2)$. Moreover, $\Aut_F(\A^2,\Delta)$ is trivial.
\item
There exists an integer $n\ge 1$ such that the group 
$\Aut(\A^2,\Delta)$ is equal to $\{g\in \J_n\mid g(\Delta)=\Delta\}$, and is thus
an algebraic subgroup of $\J_n$. Moreover, $\Aut_F(\A^2,\Delta)$ is trivial.
\end{enumerate}
\end{prop}
\begin{proof}
Denote by $I(\Delta)\subset \K[x,y]$ the ideal of polynomials vanishing on $\Delta$, and by $\overline{\Delta}\subset \A^2(\KK)$ {the closure} of $\Delta$, which is the set of points where $I(\Delta)$ vanishes. Then we obtain $\Aut(\A^2,\Delta)=\Aut(\A^2,\overline{\Delta})$ and $\Aut_F(\A^2,\Delta)=\Aut_F(\A^2,\overline{\Delta})$. We can thus replace $\Delta$ with $\overline{\Delta}$. 

If $\Delta$ is a finite union of points, we get case $i)$. Otherwise, 
$\Delta$ consists of one curve $\Gamma$ (reduced but not necessarily irreducible) and a finite number of points of $\A^2(\KK)$. Thus, 
$\Aut(\A^2,\Delta)\subset \Aut(\A^2,\Gamma)$. The result follows then from the description of $\Gamma$ and $\Aut(\A^2,\Gamma)$, given in Theorem~\ref{Thm:Main}.
\end{proof}

In the case where $\Delta$ is finite, the group $\Aut(\A^2,\Delta)$ is quite big; indeed it is  often maximal. This is the case for example when $\K=\mathbb{C}$, as pointed out to us by J.-P. Furter and P.-M. Poloni. This is a consequence of the following observation.
 
 \begin{lem}\label{lem:G}
 Let $G$ be a group acting on a set $S$. Let $\Delta\subset S$ be a finite subset of $r\ge 1$ points. Suppose that $G$ acts $2r$-transitively on $S$, and that $\lvert S\rvert >2r$. Then, 
 $$G_{\Delta}=\{g\in G\ |\ g(\Delta)=\Delta\}$$
is a maximal subgroup of $G$.
 \end{lem}
\begin{proof}Since $G_{\Delta}$ is not trivial and not equal to $G$ (because of $r$-transitivity),  it suffices to take $a\in G\setminus G_\Delta$ and to show that $a$ and $G_{\Delta}$ generate $G$. We can write $\Delta=\{x_1,\dots,x_r\}$, with $a(x_1)$, $\dots$, $a(x_k)\notin \Delta$ and $a(x_{k+1}),$ $\dots$, $a(x_r)\in \Delta$, where $1\le k\le r$. The hypotheses yield the existence of $g\in G$ that fixes $x_1,\dots,x_r,a(x_2),\dots,a(x_{k})$, and does not fix $a(x_1)$. Then, $g\in G_\Delta$ and $f=a^{-1}ga$ fixes $x_2,\dots,x_r$ but $f(x_1)\notin\Delta$. 
	
	It remains to see that any $h\in G\setminus G_\Delta$ is generated by $f$ and $G_\Delta$. We write  $\Delta=\{z_1,\dots,z_r\}$, with $h(z_1)$, $\dots$, $h(z_j)\notin \Delta$ and $h(z_{j+1}),$ $\dots$, $h(z_r)\in \Delta$, where $1\le j\le r$. Replacing $h$ with its composition with an element of $G_\Delta$ we can assume that $h(z_i)=z_i$ for $i=j+1,\dots,r$.	For $i=1,\dots,j$, we choose $g_i\in G_{\Delta}$ that sends $z_i$ onto $x_1$ and sends $h(z_i)$ onto $f(x_1)$. Then, $(g_i)^{-1} f g_i$ sends $z_i$ onto $h(z_i)$ and fixes $\Delta\setminus\{z_i\}$. Composing this element with an element of $G_\Delta$, we find an element $f_i$, generated by $f$ and $G_\Delta$, which sends $z_i$ onto $h(z_i)$ and fixes $(\Delta \cup h(\Delta))\setminus\{z_i,h(z_i)\}$. 
	
	Since  $h^{-1}f_1\cdots f_j$ belongs to $G_{\Delta}$, this achieves the proof.
%
%
%
\end{proof}

\begin{cor}\label{CorInfMax}
Assume that the ground field $\K$ is not a finite field of characteristic $2$ and let $\Delta\subset \A^n(\K)$ be a finite proper non-empty set, with $n\ge 2$. Then, $\Aut(\A^n,\Delta)$ is a maximal subgroup of $\Aut(\A^n)$.
\end{cor}
\begin{proof}
If $\K$ is infinite, we use Lemma~\ref{lem:G} and the fact that $\Aut(\AA^2)$ acts $m$-transitively on $\AA^2(\K)$ for every $m \geq 1$, which can be seen using the subgroup
\[
	\{(x_1,\dots,x_n)\mapsto (x_1+p(x_2,\dots,x_n),x_2,\dots,x_n)\mid p\in \K[x_2,\dots,x_n]\} \, ,
\] 
and  permutations of coordinates.

If $\K$ is a finite field of characteristic $>2$, the group $\Aut(\AA^2)$ acts $m$-transitively on $\AA^2(\K)$ for each $m$ by \cite{Mau01}; we can then apply Lemma~\ref{lem:G} to $\Delta$ or its complement.
	\end{proof}
\begin{cor}\label{CorFinMax}
Assume that the ground field $\K$ is a finite field of characteristic $2$ and let $\Delta\subset \A^n(\K)$ be a finite proper non-empty set, with $n\ge 2$.

Then, the group $\Aut(\A^n,\Delta)$ is a maximal subgroup of $\Aut(\A^n)$ if and only if  $\lvert \Delta\rvert\not=\frac{1}{2}\lvert \A^n(\K)\rvert $.
\end{cor}
\begin{proof}Let us write $\lvert \A^n(\K)\rvert=2m$ for some integer $m$.

 If $\lvert \Delta\rvert<m$, the fact that $\Aut(\A^n,\Delta)$ is a maximal subgroup of $\Aut(\A^n)$ follows from Lemma~\ref{lem:G} and from the fact that the action of $\Aut(\A^n)$ on the $2m$ points of $\A^n(\K)$ give all even permutations (see \cite{Mau01}), and thus acts $(2m-2)$-transitively.  If $\lvert \Delta\rvert>m$, we exchange $\Delta$ with its complement.

 If $\lvert \Delta\rvert=m$, there exists an automorphism $\varphi$ of $\A^n$ exchanging $\Delta$ with its complement (by the result of \cite{Mau01} cited before). Denoting by $H$ the group generated by $\Aut(\A^n,\Delta)$ and $\varphi$, we have $\Aut(\A^n,\Delta)\subsetneq H\subsetneq\Aut(\A^n)$.
\end{proof}
\begin{rem}
Corollaries~\ref{CorInfMax} and~\ref{CorFinMax} raise the question of describing  all maximal subgroups of $\Aut(\A^n)$ in general.
\end{rem}

\subsection{Generalisation to higher dimension}
\label{GeneralizationsHigherDim.subsec}
Let us show with an example 
that in higher dimension, the hypersurfaces 
$X\subset \A^n$ such that $\Aut(\A^n,X)$ is not an algebraic group are not as simple as in dimension $n = 2$.

\begin{exa}
Let $X\subset \A^3$ {be} the hypersurface {with} equation $xy=f(z)$, for some polynomial $f\in \K[z]$. Then, $\Aut(\A^3,X)$ contains {the group}
$$\left\{(x,y,z)\mapsto \left(x,y+\frac{f(z+xq(x))-f(z)}{x},z+xq(x)\right)\mid q\in \K[x] \right\}\simeq \K[x]$$
and thus {it} is not an algebraic group.
\end{exa}

A possible generalisation of Theorem~\ref{Thm:Main} would be to show that every hypersurface $X\subset \A^3$ such that $\Aut(\A^3,X)$ is not an algebraic group admits an $\A^1$-fibration $X\to \A^1$ given by a coordinate projection.

\section{Classification of the possible group actions and the proof of Theorem~\ref{Thm:Class}}
\label{Sec:Class}
This section is devoted to the proof of Theorem~\ref{Thm:Class}, which describes more precisely the curves and groups appearing in Theorem~\ref{Thm:Main}, in the case where the ground field $\K$ is perfect, and where the curve is geometrically irreducible.
\subsection{The possibilities for $\Gamma$ and $\Aut(\A^2,\Gamma)$}
\label{PossibilitiesFor.subsec}
\begin{lem}
	\lab{LEM:AutoCurve}
	Let $\Gamma$ be an affine geometrically irreducible curve, defined over a perfect field $\K$. The group $\Aut(\Gamma)$ is an affine algebraic group. If it has positive dimension, one of the 
	following holds:
	\begin{enumerate}[$i)$]
		\item $\Gamma \simeq \AA^1$;
		\item The curve $\Gamma$ is unicuspidal with normalization $\A^1$;
		\item Over the algebraic closure $\KK$, 
			$\Gamma$ is isomorphic to $\AA^1 \setminus \{ 0 \}$.
	\end{enumerate}
\end{lem}

\begin{proof}

	Let $\tilde{\Gamma}$ be the normalization of $\Gamma$, which can be viewed as an 
	open subset of
	a smooth projective curve $X$, defined over~$\K$. Let
	$(X\setminus \tilde{\Gamma})(\KK)=\{x_1,\dots,x_r\}$ be its 
	complement. The points $x_1,\dots,x_r$ are not 
	necessarily all defined 
	over~$\K$ (however, the union is invariant by 
	the Galois group $\Gal(\KK/\K)$). Moreover, we denote by 
	$x_{r+1},\dots,x_m$ the $\KK$-points of $\tilde{\Gamma}$ which are send onto the 
	singular $\KK$-points of $\Gamma$. As before, not all are necessarily defined over~$\K$.
	
	This yields a natural inclusion
	\[
		\Aut(\Gamma)\subseteq  \{g\in \Aut(X) \mid g(\{x_1,\dots,x_m\})=
		\{x_1,\dots,x_m\}\},
	\] 
	and a group homomorphism $\Aut(\Gamma)\to \mathrm{Sym}_m$. 
	The kernel is of the same dimension as $\Aut(\Gamma)$. 
	
	Let us recall some classical facts on automorphisms of smooth 
	projective curves.
	If the genus of $X$ is at least $2$, then $\Aut(X)$ is finite. 
	If the genus is $1$, the subgroup of $\Aut(X)$ that fixes a point is 
	also finite. If $X$ is rational, the subgroup of automorphisms fixing three 
	points is trivial. 
	
	If the dimension of $\Aut(X)$ is positive, we obtain that $X$ is rational and that $1\le r\le m\le 2$.
	
	If $r=m=1$, then $\Gamma=\tilde{\Gamma}\simeq \A^1$ (every form of the affine line over a perfect field is trivial, see \cite{Rus70}).
	
	If $r=1$ and $m=2$, then $\tilde{\Gamma}\simeq \A^1$ and $\Gamma$ is 
	{a unicuspidal} curve.
	
	If $r=m=2$, then $\Gamma=\tilde{\Gamma}$ is smooth and isomorphic, over $\KK$, to $\A^1\setminus \{0\}$.
\end{proof}
\begin{rem}
Lemma~\ref{LEM:AutoCurve} is false over a non-perfect field, since there are non-trivial forms of the affine line and its additive group; see  \cite{Rus70} for a classification of such curves.
\end{rem}

\begin{lem}\label{Lem:Dim12}
	Assume that $\K$ is a perfect field.
	Let $\Gamma \subseteq \AA^2$ be a closed  geometrically irreducible curve and assume that
	$\Aut(\AA^2, \Gamma)$ is an algebraic group 
	of positive dimension $($over $\KK)$.  Then, $\Aut(\A^2,\Gamma)$ 
	contains a closed subgroup $G$ defined over~$\K$, which is either
	\begin{enumerate}[$i)$]
	\item
	a one-dimensional torus, i.e.\;isomorphic to 
	the one-dimensional multiplicative group $\GG_m$ over $\KK$,
	\item
	or isomorphic to the one-dimensional additive group
	$\GG_a$ over~$\K$.
	  \end{enumerate}
\end{lem}

\begin{proof}
	The algebraic group $\Aut(\AA^2, \Gamma)$ is isomorphic to a closed subgroup of $\Aut(\Gamma)$ (Theorem~\ref{Thm:Main}). This gives three possibilities for $\Gamma$, according to  Lemma~\ref{LEM:AutoCurve}.
	
	$a)$ $\Gamma\simeq \A^1$, and we get
	$\Aut(\Gamma)\simeq \GG_a\rtimes \GG_m$.
	
	$b)$ $\Gamma$ is a unicuspidal curve, in which case 
	$\Aut(\Gamma)$ is a torus.
	
	$c)$ $\Gamma$ is isomorphic to $\A^1\setminus\{0\}$ over $\KK$. The  
	connected component of the identity $\Aut(\AA^2, \Gamma)^0$ is a 
	connected algebraic group, defined over $\K$, which is a torus.
\end{proof}

The two possibilities given by Lemma~\ref{Lem:Dim12} are described respectively in $\S\ref{SecTorus}$ and $\S\ref{SecGa}$. We will in particular show that the second case does not occur.
\subsection{Torus actions}
\label{SecTorus}
\begin{lem}\label{Lem:GL2}
	Assume that $\K$ is a perfect field.
	Let $\Gamma \subseteq \AA^2$ be a closed geometrically irreducible curve
	and assume that
	$\Aut(\AA^2, \Gamma)$ is an algebraic group that 
	contains a closed one-dimensional torus $T$.
	Then there exists an automorphism $\varphi \colon \AA^2 \to \AA^2$ 
	such that $\varphi \circ T \circ \varphi^{-1} \subseteq \GL(2,\K)$.
\end{lem}
\begin{proof}
	By Corollary \ref{COR:Main} we can assume that either
	$\Aut(\AA^2, \Gamma) \subseteq \Aut(\FF_n, E_n \cup L_{\FF_n})$ or 
	$\Aut(\AA^2, \Gamma) \subseteq \Aut(\PP^2, L_{\PP^2})$.

	Moreover, we can assume that $T$ is defined over~$\K$ and $T$ is isomorphic to $\mathbb{G}_m$ over $\KK$ (Lemma~\ref{Lem:Dim12}).
	
	Assume $\Aut(\AA^2, \Gamma) \subseteq \Aut(\FF_n, E_n \cup L_{\FF_n})$. 
	Thus $T$ acts on $L_{\FF_n}$. If this action is trivial, then there 
	exists a fixed point on $L_{\FF_n} \setminus E_n$ 
	that is defined over~$\K$. Otherwise,
	$T$ has exactly two fixed points on $L_{\FF_n} \simeq \PP^1$ defined 
	over $\KK$. As $L_{\FF_n} \cap E_n$ is a fixed point
	of the $T$-action that is defined over~$\K$, there is a fixed point on
	$L_{\FF_n} \setminus E_n$ that is defined over~$\K$. Thus by performing
	elementary links, we can assume that 
	$T \subseteq \Aut(\FF_1, E_1 \cup L_{\FF_1})$. But $T$  preserves
	the exceptional divisor $E_1$ and therefore 
	$\varphi \circ T \circ \varphi^{-1} \subseteq \Aut(\PP^2, L_{\PP^2})$, 
	where $\varphi \colon \FF_1 \to \PP^2$ denotes a link of type III.
	
	Thus we are left over with the case $T \subseteq \Aut(\PP^2, L_{\PP^2})$.
	It is enough to show that the induced action of $T$ on 
	$\AA^2 = \PP^2Ê\setminus L_{\PP^2}$ has a fixed point
	that is defined over~$\K$. The set of points of $\AA^2(\KK)$ that are fixed by $T(\KK)$ consists either of one affine line or one point. This set is invariant by the action of the Galois group 
$\textrm{Gal}(\KK / \K)$, and is thus defined over $\K$. Using again the fact that every form of the affine line over a perfect field is trivial (see \cite{Rus70}), we find a $\K$-point of $\A^2$ fixed by $T$.
\end{proof}

\begin{prop}\label{Prop:Torus}
	Assume that $\K$ is perfect. Let $\Gamma \subseteq \AA^2$ be a closed geometrically irreducible curve and assume that
	$\Aut(\AA^2, \Gamma)$ contains a closed one-dimensional torus. After conjugation by an automorphism of $\AA^2$, the curve $\Gamma$ 
	
	\begin{enumerate}[$i)$]
		\item has equation $x=0$, or
		\item  has equation $x^b= \lambda y^a$ where $a, b>1$ are 
			coprime integers and $\lambda \in \K^\ast$, or
		\item  has equation $x^b y^a = \lambda$ where $a, b \geq 1$ are 
			coprime integers and $\lambda\in \K^{*}$, or
		\item
			has equation  $\lambda x^2+\nu y^2=1$, where 
			$\lambda,\nu\in \K^{*}$, $-\lambda\nu$ is not a square in $\K$ and $\char(\K)\not=2$, or
		\item
			has equation  $ x^2+\mu xy+y^2=1$, where 
			$\mu \in \K^{*}$, $x^2+\mu x+1$ has no root in $\K$ and $\char(\K)=2$.
	\end{enumerate}
	
	Moreover, the group $\Aut(\AA^2,\Gamma)$ is respectively equal to 
	\begin{enumerate}[$i)$]
	\item
	$\{(x,y)\mapsto (ax,by+P(x))\mid a,b\in \K^{*}, P\in \K[x]\}\simeq \K[x]\rtimes 
	{(\K^{*})^2}$;
	\item
	$\{(x,y)\mapsto (t^ax,t^by)\ |\ t\in \K^{*}\}\simeq \K^{*}$;
	\item
	$\{(x,y)\mapsto (t^ax,t^{-b}y)\ |\ t\in \K^{*}\}\simeq \K^{*}$ if $(a,b)\not=(1,1)$; \\
	$\{(x,y)\mapsto (tx,t^{-1}y)\ |\ t\in \K^{*}\}\cup \{(x,y)\mapsto (ty,t^{-1}x)\ |\ t\in \K^{*}\} \simeq \K^{*}\rtimes \mathbb{Z}/2\mathbb{Z}$ if $(a,b)=(1,1)$;
	\item $T\rtimes \langle \sigma\rangle\simeq 
	T\rtimes \mathbb{Z}/2\mathbb{Z}$, where $T,\{\sigma\}\subset \GL(2,\K)$ 
	are given by
	$$T=\left.\left\{\left(\begin{array}{cc} a& -\nu  b\\ \lambda b&  a\end{array}\right)\right| a^2+\lambda\nu b^2=1  \right\}, \ \sigma=\left(\begin{array}{cc} 1& 0\\ 0& -1\end{array}\right).$$
	Moreover $T$ is a torus, which is not $\K$-split.
	\item
	$T\rtimes \langle \sigma\rangle\simeq T\rtimes \mathbb{Z}/2\mathbb{Z}$, where $T,\{\sigma\}\subset \GL(2,\K)$ are given by
	$$T=\left.\left\{\left(\begin{array}{cc} a&  b\\  b& a+\mu b\end{array}\right)\right| a^2+\mu ab+b^2=1 \right\}, \ \sigma=\left(\begin{array}{cc} 1& \mu \\ 0& 1\end{array}\right).$$
	Moreover $T$ is a torus which is isomorphic to $\Gamma$ $($and which is not $\K$-split$)$.
	\end{enumerate}
\end{prop}

\begin{proof}
By Theorem~\ref{Thm:Main}, we can suppose that $\Aut(\AA^2, \Gamma)$ is an algebraic group, and using Lemma~\ref{Lem:GL2}, we can  {moreover assume}
that 
	$\Gamma$ is preserved by a torus $T\subset\GL(2,\K)$.
	
	There exists an element $\psi\in\GL(2,\KK)$ which conjugates $T(\KK)$ to 
	\[
		\lambda \mapsto 
		\begin{pmatrix}
			\lambda^a & 0 \\
			0 & \lambda^b
		\end{pmatrix}
	\]
	for integers $a$, $b$ with $(a, b) \neq (0, 0)$ and $a$, $b$ are coprime. 
	If $a$ or $b$ is equal to zero, then it follows that $\Gamma$
	is a line. Hence we can assume that $a$ and $b$ are non-zero.
	Now, let $(x_0, y_0) \in \psi(\Gamma(\KK))$, such that $(x_0, y_0) \neq (0, 0)$.
	If $x_0$ or $y_0$ is zero, then $\Gamma$ is again a line. Hence, we may assume
	that $x_0 \neq 0 \neq y_0$. By symmetry, we can assume that $b>0$. Then the equation of 
	$\psi(\Gamma)$ is 
	\begin{center}\begin{tabular}{ll}
	$y_0^a x^b - x_0^b y^a=0$ &if $a>0 \,$,\\
	$y_0^a x^b y^{-a} - x_0^b=0$ &if $a<0 \,$.\end{tabular}\end{center}
	In the first case, we can assume that $a>1$ or $b>1$, otherwise 
	the curve is a line. 
	This implies that $a\not=b$, since both are coprime, and we can thus assume that $a>b$. The  Galois group $\Gal(\KK/\K)$ fixes the unique point at infinity, which is then defined over $\K$. Hence, we can assume that the unique point of the closure of $\Gamma$ at infinity is the direction $y=0$, and that the equation of $\Gamma$ is the polynomial $(\alpha x+\beta y)^b-(\gamma y)^a\in \K[x,y]$, for some $\alpha,\beta,\gamma\in \KK$, $\alpha\gamma\not=0$. If $\beta/\alpha\in \K$, we make a change of coordinates $(x,y)\mapsto (x-\frac{\beta}{\alpha} y,y)$ and obtain an equation of the form $x^b-\lambda y^a$ for some $\lambda\in \K$, as desired. It remains to see that $\beta/\alpha$ always belong to $\K$. If the characteristic of $\K$ is zero, we develop $(\alpha x+\beta y)^b$ and divide the coefficient of $x^{b-1}y$ by the coefficient of $x^b$. If the characteristic of $\K$ is $p>0$, we write $b=qm$, where $q$ is a power of $p$ and $p$ does not divide $m$. Developing, we find
	$$(\alpha x+\beta y)^b=(\alpha^q x^q+\beta^qy^q)^m=\alpha^{b} x^b+m\alpha^{b-q}\beta^qx^{b-q}y^q+\dots$$
	hence $(\beta/\alpha)^q\in \K$, which implies, since $\K$ is a perfect field, that $\beta/\alpha\in \K$.
	
		In the second case ($a<0$), the closure of the 
	curve $\psi(\Gamma)$ has two points at infinity in $\p^2$. If $a\not=-b$, the two points have 
	different multiplicities. 
	In consequence the  {Galois group $\Gal(\KK/\K)$} has to fix the two 
	points at infinity of 
	$\Gamma$, which are thus defined over~$\K$. We can assume that 
	these points correspond to the 
	directions $x=0$ and $y=0$, and that $\psi\in \GL(2,\KK)$ is diagonal. 
	The equation of 
	$\Gamma$ is then of the form $x^by^a-\lambda$ for some 
	$\lambda\in \K^{*}$, and we get case $iii)$.
	The only remaining case is when 
	$(-a,b)=(1,1)$ and the 
	two points at infinity of $\Gamma$ are exchanged by $\Gal(\KK/\K)$. 
	The equation of 
	$\psi(\Gamma)$ being of the form $xy=x_0 y_0$, the equation 
	of $\Gamma$ is of 
	the form 
	\[
		\lambda x^2+\mu xy+\nu y^2=1 \, ,
	\]
	where $\lambda,\mu,\nu\in\K$, and $\lambda x^2+\mu xy+\nu y^2\in \K[x,y]$ is irreducible. When the characteristic of $\K$ is not $2$, we can make a change of coordinates $(x,y)\mapsto (x-\frac{\mu}{2 \lambda} y,y)$ and assume that $\mu=0$. The two points at infinity are thus given by $\lambda x^2+\nu y^2=0$. Because the two points are not defined over $\K$, we find that $-\nu\lambda$ is not a square in $\K$, and obtain $iv)$. If the characteristic of $\K$ is $2$, the elements $\lambda,\nu$ are squares in $\K$, since $\K$ is perfect. Making a diagonal change of variables, we can then assume that $\lambda=\nu=1$. Then $x^2+\mu x+1$ has no root in $\K$, and we get $v)$.

	It remains to prove that $\Aut(\A^2,\Gamma)$ has the desired form. 
	Case $i)$ is a direct calculation.
	 In cases $ii), iii),iv), v)$, it can be checked that the group given are contained in $\Aut(\A^2,\Gamma)$, so we just need to see that there is no other automorphism. The group homomorphism $\Aut(\A^2,\Gamma)\to \Aut(\Gamma)$ being injective, the only case to consider is the curve $x^by^a-\lambda=0$, with $(a,b)\not=(1,1)$, and to prove that there is no automorphism of $\A^2$ inducing on $\Gamma$ an``exchange" of the two points at infinity. These two points are {$p_1=(1:0:0)\in \p^2$ and $p_2=(0:1:0) \in \p^2$}, and have multiplicity $a$ and $b$ respectively on $\Gamma$, and $\h_{\Gamma}(p_1) = \h_{\Gamma}(p_2) > 0$ if $b>1$. 
We can moreover assume $a>b\ge 1$. The hypothetic automorphism extends to a birational map $\varphi\colon\p^2\dasharrow \p^2$ which is not an automorphism of $\p^2$, and thus decomposes into a sequence of elementary links $\varphi=\varphi_m \cdots  \varphi_1$ (Proposition~\ref{PROP:ExistenceDecomposition}). By Corollary~\ref{COR:Key}, we have $b=1$ and the point blown-up by the first link $\varphi_1\colon (\p^2,L_{\p^2})\dasharrow (\FF_1,E_1\cup L_{\FF_1})$ is $p_1$. 
Looking at the equation of $\Gamma$ in $\A^2$, we can describe the closure  of $\Gamma$ on $\FF_1$. This projective curve is smooth, intersects transversally $L_{\FF_1}$ in one point  away from $E_1$, corresponding to $p_2$, and intersects $E_1$ in one point $q_1$, with multiplicity $a$, corresponding to $p_1$; this latter is moreover not on $L_{\FF_1}$. In consequence, the next links of type $\mathrm{II}$ do not affect the point $q_1$, and after the first link of type $\mathrm{III}$, the image of the curve is singular at a point of $\p^2$ with multiplicity $\ge a$ and height $1$. This point being not the base-point of the next elementary links, the image of $\Gamma$ by $\varphi$ has again a singular point, corresponding to the image of $p_1$. It is thus not possible to ``exchange" $p_1$ and $p_2$.
\end{proof}

\subsection{$\GG_a$-actions}\label{SecGa}
The classification of all  $\GG_a$-actions on $\AA^2$ is known when the ground 
{field} $\K$ is of characteristic $0$ \cite{Re68} or algebraically closed of positive characteristic \cite{Mi71}. The following lemma gives the generalisation of \name{Miyanishi}'s result to the case where $\K$ is perfect. The proof is probably known to the specialists, we include it for the sake of completeness and lack of reference.
\begin{prop}\label{Prop:MiyaPerfect}
	Assume that $\K$ is {a} perfect field of characteristic $p > 0$. 
	Then every $\GG_a$-action on $\AA^2$
	that is defined over~$\K$ has the form
	\[
		(t, x, y) \mapsto (x, y + tf_0(x) + t^p f_1(x) + \ldots + t^{p^n} f_n(x))
	\]
	{up to conjugation with a $\K$-automorphism of $\AA^2$.}
\end{prop}

\begin{proof}
	We fix a non-trivial $\GG_a$-action on $\AA^2$, that is defined over $\K$.
	By \cite{Mi71} it follows that the $\GG_a$-action has the claimed form 
	over the algebraic closure $\KK$ 
	{(up to conjugation with a $\KK$-automorphism of $\AA^2$)}. 
	Thus there exists a {$\GG_a$-invariant}
	polynomial
	$f \in \KK[x, y]$ which is a variable of $\KK[x, y]$, i.e.\;which admits $g \in \KK[x, y]$ such that
	$\KK[f, g] = \KK[x, y]$. Let $G := \Gal(\KK/\K)$ be the Galois group. 
	Now, $\KK[f]$ is the subring
	of $\GG_a$-invariant polynomials of $\KK[x, y]$ and since the $\GG_a$-action is defined
	over~$\K$, the subspace $\KK[f]$ is invariant under $G$.
	Thus, the action of $G$ on $\KK[f, g] = \KK[x, y]$ is given by
	\[
		\sigma(f)=a_\sigma f + c_\sigma \, , \quad \sigma(g) = b_\sigma g + d_\sigma
	\]
	where $a_\sigma, b_\sigma \in \KK^\ast$, $c_\sigma \in \KK$ and $d_\sigma \in \KK[f]$.
	{It} is enough to show that the 1-cocycle
	\[
		G {\to} \J_n(\KK) \simeq (\KK^\ast)^2 \ltimes (\KK \ltimes \KK[f]) \, , \quad
		\sigma \mapsto (a_\sigma, b_\sigma, c_\sigma, d_\sigma)
	\]
	is a 1-coboundary. The vanishing of $\H^1(G, J_n(\KK))$ 
	follows from the vanishing of $\H^1(G, \KK^\ast)$ (see \cite[(6.2.1) Theorem]{NSW00})
	and from the vanishing of $\H^1(G, \KK[f]) = \varinjlim_n \H^1(G, \KK[f]_{\le n})$ 
	(see  \cite[(1.5.1) Proposition]{NSW00} and \cite[(6.1.1) Theorem]{NSW00}) by using exact sequences
	(here $\KK[f]_{\le n}$ denotes the polynomials in $f$ of degree $\leq n$).
\end{proof}

\begin{lem}\label{Lem:GaAffLine}
	Assume that $\K$ is perfect.
	Let $\Gamma \subseteq \AA^2$ be a closed geometrically irreducible curve that is defined over~$\K$ and
	assume that it is preserved under
	a non-trivial $\GG_a$-action $($defined over $\K)$. Then there exists an automorphism 
	$\varphi \colon \AA^2 \to \AA^2$
	such that $\varphi(\Gamma)$ is an affine line in $\AA^2$.
\end{lem}

\begin{proof}
	By Proposition~\ref{Prop:MiyaPerfect} (in case $\char(\K)  = p > 0$) and by 
	\cite{Re68} (in case $\char(\K) = 0$) we can conjugate the action to an action of the form {$(t, x,y)\mapsto (x,y+p(t,x))$} where $p\in \K[t,x]$ is a {non-zero}
polynomial. Hence, every geometrically irreducible $\GG_a$-invariant 
curve is a line in $\A^2$.
\end{proof}
Lemma~\ref{Lem:GaAffLine} implies that {the} second case of Lemma~\ref{Lem:Dim12} does not 
{occur}. The proof of Theorem~$\ref{Thm:Class}$ is now clear:

\begin{proof}[Proof of Theorem~$\ref{Thm:Class}$]
By Theorem~\ref{Thm:Main}, either $\Gamma$ is a line or  $\Aut(\A^2,\Gamma)$ is an algebraic subgroup of $\Aff(\A^2)$ or $\J_n$ for some $n\ge 1$. If $\Aut(\A^2,\Gamma)$ is an algebraic group of positive dimension, it contains a closed one-dimensional torus by Lemmas~\ref{Lem:Dim12} and~\ref{Lem:GaAffLine}. The description of the possible cases follows then from Proposition~\ref{Prop:Torus}.
\end{proof}

\subsection{The case of finite groups}
\label{FiniteGrps.subsec}
There are plenty of examples where $\Aut(\A^2,\Gamma)$ is finite. The simplest way to get such examples {is} to take a finite subgroup  $G\subset\Aut(\A^2)$ and to look for invariant curves. Since $G$ has a finite action on $\K[x,y]$, one can find a lot of invariant polynomials, and in practice most of them are irreducible.

In characteristic zero, the group $G$ is reductive and thus contained, up to conjugation, in $\GL_2$ (see \cite{Kam79}). In positive characteristic, there are however plenty of non-linearisable subgroups of $\Aut(\A^2)$, and, as far as we know, there is for the moment no classification of the conjugacy classes of such subgroups.

\begin{exa}
Let $\K$ be of characteristic $p>0$. For any integer $a>1$,
$$\varphi\colon (x,y)\mapsto (x,y+x^{a})$$ is a non-linearisable automorphism of order $p$ of $\A^2$, which preserves the family of curves of the form 
\[
	y^p-yx^{a(p-1)}=q(x) \, ,
\]
where $q\in \K[x]$.
\end{exa}

\end{document}